\documentclass[12pt]{article}
\title{Borel complexity of modules}
\author{ Michael C.\
Laskowski  
and Danielle S.\ Ulrich
 \thanks{Both authors partially supported
by NSF grants DMS-1855789 and DMS-2154101.}
\\
Department of Mathematics\\University of Maryland
}

\usepackage{amsmath, amssymb, amsthm, enumerate, fullpage, bm}
\usepackage{verbatim, enumitem, bbm, etoolbox}
\usepackage{latexsym}
\def\bp{\par{\bf Proof.}$\ \ $}

\def\includeE#1{{\lhook\kern-3.5pt\joinrel\smash{
    \mathop{\longrightarrow}\limits^{#1}}}}

\def\efor/{Example~\ref{E4}}

\def\BL/{Baldwin--Lachlan}
\def\Bu/{Buechler}
\def\Hr/{Hrushovski}
\def\lm/{locally modular}
\def\wm/{weakly minimal}
\def\nm/{non--modular}
\def\ss/{superstable}
\def\ud/{unidimensional}
\def\sm/{strongly minimal}
\def\E{{\cal E}}

\def\abar{\overline{a}}

\def\cbar{\overline{c}}

\def\hbar{\overline{h}}

\def\xbar{\overline{x}}

\def\Mbar{\overline{M}}

\def\cl{{\rm cl}}

\def\tr/{trivial}
\def\nt/{non--trivial}
\def\st/{strong type}

\def\abar{\bar{a}}

\def\cbar{\bar{c}}

\def\phi{\varphi}

\def\FF{{\bf F}}

\def\Z{{\mathbb Z}}

\def\Fa0{{\FF^a_{\aleph_0}}}

\def\bp{{\bf Proof.}\quad}

\def\<{\langle}
\def\>{\rangle}

\def\FF{{\mathbb F}}
\def\pp{{\mathfrak p}}
\def\K{{\bf K}}
\def\aa{{\bf a}}
\def\bb{{\bf b}}
\def\cc{{\bf c}}
\def\dd{{\bf d}}
\def\ee{{\bf e}}
\def\gg{{\bf g}}
\def\jj{{\bf j}}
\def\Rhat{\hat{R}}
\def\eps{\epsilon}
\def\m{{\mathfrak m}}

\def\O{{\cal O}}
\def\Vbar{\overline{V}}
\def\Mod{\mbox{Mod}}
\def\L{{\mathcal L}}

\newtheorem{Theorem}{Theorem}[section]
\newtheorem{Proposition}[Theorem]{Proposition}
\newtheorem{Definition}[Theorem]{Definition}

\newtheorem{Remark}[Theorem]{Remark}

\newtheorem{Lemma}[Theorem]{Lemma}
\newtheorem{Corollary}[Theorem]{Corollary}

\newtheorem{Fact}[Theorem]{Fact}

\begin{document}

\date{\today}

\maketitle

\begin{abstract}  We prove that for a countable, commutative ring $R$, the class of countable $R$-modules either has only countably many isomorphism types, or else it is Borel complete.
The machinery gives a succinct proof of the Borel completeness of \mbox{TFAB}, the class of torsion-free abelian groups.  We also prove that for any countable ring $R$, both the class 
of left $R$-modules
endowed
with an endomorphism and the class of left $R$-modules with four named submodules are Borel complete.
\end{abstract}

\section{Introduction}  The notion of {\em Borel reducibility} was introduced by Friedman and Stanley~\cite{FS} to measure the complexity of 
quotients of Polish spaces by natural equivalence relations.  In model theory, the natural context is the Polish space $Str_\L$ of all structures with universe
$\omega$ in a countable language $\L$.  With the natural topology on $Str_\L$, for any $\L$-theory $T \in \mathcal{L}_{\omega_1 \omega}$, the space $\mbox{Mod}(T)$ of models of $T$ forms a Borel subset of $Str_\L$ and the equivalence relation of isomorphism, considered as a subset of $\Mod(T)^2$, is analytic.
Given an $L$-theory $T$ and an $\L'$-theory $S$, Friedman and Stanley say {\em $\Mod(T)$ is Borel reducible to $Mod(S)$,}
written $T\le_B S$, if there is a Borel function $f:\Mod(T)\rightarrow \Mod(S)$ such that for all models $M,N\models T$, $f(M)\cong_{\L'} f(N)$ if and only if $M\cong_\L N$.
Say the theories $T,S$ are {\em Borel equivalent} if both $T\le_B S$ and $S\le_B T$.  Thus, we get a notion of {\em Borel cardinality} comparing different classes of isomorphism types of countable structures.
Friedman and Stanley proved that there is a maximal class of theories with respect to Borel reducibility.  Theories in this maximal class are called {\em Borel complete}.
It is easily seen that if $T$ is Borel complete, then there are $2^{\aleph_0}$ isomorphism types of countable models of $T$. But, in general, Borel cardinality is a finer measure of complexity
than counting isomorphism types.  There are many theories with $2^{\aleph_0}$ non-isomorphic countable models that are Borel incomparable, but our main theorem here says that this distinction does not occur among classes of $R$-modules for any countable, commutative ring $R$.

%
%


In this paper, we investigate the Borel cardinality of theories of  left $R$-modules as we vary the countable ring $R$. We will only ever consider rings with unit, and when $R$ is noncommutative we will only ever consider left $R$-modules. We obtain results for several non-commutative rings, but our main theorem completely solves the problem in the commutative case (in which case the left versus right distinctions vanish). The following was already known, see e.g., Chapter 11 of \cite{Prest}.

\begin{Theorem}  \label{classical}
	Let $R$ be any countable, commutative ring. Then the following are equivalent:
	
	\begin{enumerate}
		\item $R$ is an Artinian principal ideal ring;
		\item $R$ is of finite representation type (every module is a direct sum of indecomposable modules, and there are only finitely many indecomposable modules up to isomorphism);
		\item For every $R$-module $M$, $\mbox{Th}(M)$ is $\omega$-stable;
		\item There are fewer than continuum many isomorphism classes of countable $R$-modules;
		\item There are only countably many isomorphism classes of countable $R$-modules.
	\end{enumerate}
\end{Theorem}

Our main theorem adds the failure of Borel completeness to the above list.

\begin{Theorem} \label{big} Let $R$ be any countable, commutative ring (with unit).  Then either the theory of $R$-modules is Borel complete, or
	else there are only countably many isomorphism types of countable $R$-modules.
\end{Theorem}

\medskip

The proof of Theorem~\ref{big} splits into two disjoint pieces.  The first, which does not require $R$ to be commutative,
concerns classes of {\em tagged  $R$-modules}.
 This concept was introduced by G\"obel and Shelah in \cite{GS} under a different name.  

\begin{Definition} \label{above} {\em  Let $R$ be any countable ring (not necessarily commutative).  A {\em tagged (left)  $R$-module} is a left $R$-module $M$, expanded by adding unary predicates for 
countably many left $R$-submodules of $M$.
}
\end{Definition}  

In Section~\ref{taggedsection} we prove 

\begin{Theorem} \label{taggedBC}  For any countable ring $R$ (not necessarily commutative) the theory of tagged  $R$-modules is Borel complete.
\end{Theorem}

One of the celebrated open questions from \cite{FS} was whether or not \mbox{TFAB}, the theory of torsion-free abelian groups, is Borel complete.  Modern progress on this question was made by Shelah and the second author in \cite{ShU}.  There, they prove that the class of tagged abelian groups (i.e., tagged $\Z$-modules in Definition~\ref{above}) is Borel equivalent to \mbox{TFAB}.  Thus, a positive answer to Friedman and Stanley's question follows directly from Theorem~\ref{taggedBC}.  In particular, given the results in \cite{ShU}, none of the Borel reductions given in Section~\ref{transfer} here are needed to prove that \mbox{TFAB} is Borel complete.
Alternatively, see Remark~\ref{integral} below.
Prior to this paper, Paolini and Shelah~\cite{PS}  announced a proof that 
\mbox{TFAB} is Borel complete, but their arguments are extremely intricate and earlier versions have had some  shortcomings.

Following this, in Section~\ref{transfer} we give several Borel reductions that are used to prove Borel completeness of $R$-modules when $R$ has some sort of `defect':
\begin{itemize}

\item (Theorem~\ref{A}) If a countable ring $R$ has a central, non-zero divisor, non-unit $r\in R$, then the theory of left $R$-modules is Borel complete.

\item (Theorem~\ref{B}) If a countable ring $R$ has central elements $x,y\in R$ with disjoint principal ideals $(x)\cap(y)=0$,
and if the two-sided ideal $I=\mbox{Ann}(x)+\mbox{Ann}(y)$ is proper, i.e., $1\not\in I$,
then the theory of left $R$-modules is Borel complete.
\item (Theorem~\ref{C}) If a countable, commutative ring $R$ has a descending chain $(I_n:n\in\omega)$ of annihilator ideals with $\bigcap_{n\in\omega} I_n=0$,
then the theory of $R$-modules is Borel complete.

\end{itemize}

\begin{Remark}  \label{integral} {\em As a special case of Theorem~\ref{A}, if $R$ is an integral domain that is not a field, then the theory of $R$-modules is Borel complete (take $r$ to be any nonzero nonunit). Moreover, if $R = \mathbb{Z}$ (so $R$-modules are the same as abelian groups) and if we take $r$ to be a prime, then the Borel reduction given by Theorem~\ref{A} actually only produces torsion-free groups, so we get a self-contained proof that \mbox{TFAB} is Borel complete.
}
\end{Remark}

Finally, in Section~\ref{TWO} we  use all of the above results to prove Theorem~\ref{big}, that if a countable, commutative ring $R$ is not an Artinian, principal ideal ring, then the
theory of $R$-modules is Borel complete.  



\medskip

We observe some interesting consequences of Theorem~\ref{A}:

\begin{Corollary} \label{R[x]} For any countable ring $R$, the theory of $R[x]$-modules is Borel complete.
\end{Corollary}

\begin{proof}  The indeterminate $x\in R[x]$ is a central, non-zero divisor, non-unit, so we are done by Theorem~\ref{A}. \qed
\end{proof}

\begin{Definition}  {\em  For $R$ any countable ring, we consider the class of all {\em left $R$-endomorphism structures} $(V,T)$, i.e. the class of left $R$-modules $V$ equipped with an endomorphism $T: V \to V$.
}
\end{Definition}

\begin{Corollary} \label{endo}  For any countable ring $R$, the theory of left $R$-endomorphism structures is Borel complete.
\end{Corollary}

\begin{proof}  Fix any countable ring $R$.  By Corollary~\ref{R[x]}, it is enough to Borel reduce the class of countable left $R[x]$-modules to the class of left $R$-endomorphism structures.
Send the left $R[x]$-module $V$ to the pair $(V', T)$, where $V'$ is the reduct of $V$ to an $R$-module, and $T(a) := xa$. Easily, this works.  \qed
\end{proof}

From this we obtain the following sharpening of Theorem~\ref{taggedBC}:

\begin{Corollary}  For any countable ring $R$, the theory of left $R$-modules with 4 distinguished submodules is Borel complete.
\end{Corollary}

\begin{proof}  Fix any countable ring $R$.  In light of Corollary~\ref{endo} it suffices to define a Borel reduction from left $R$-endomorphism structures $(V,T)$ to left $R$-modules $(W,U_0,U_1,U_2,U_3)$
	with four distinguished submodules.  But this is easy.  Given $(V,T)$, put $W=V\times V$ as a left $R$-module and interpret the $U_n$'s as:
	$U_0:=V \times 0$; $U_1=0 \times V$; $U_2:=\{(a,a):a\in V\}$ (the diagonal); and $U_3:=\{(a,T(a)):a\in V\}$ (the graph of $T$). \qed
\end{proof}

It is not hard to show that if $R$ is a field, then $R$ adjoined by three distinguished submodules is not Borel complete, in fact it has only countably many countable models. This  follows from the representation theory of quivers, in particular the subspace problem;  see e.g., Theorem 13.A and the preceding discussion in \cite{Prest}.

We do not know whether a dichotomy similar to Theorem~\ref{big} holds for countable, non-commutative rings $R$.  Other than the cases covered by Theorems~\ref{A} and \ref{B}, we do not have tools for determining Borel completeness.  In particular, we do not know of a similar dichotomy for quasi-simple rings (i.e., rings with no proper 2-sided ideals).

 \section{Tagged left $R$-modules are Borel complete}  \label{taggedsection}
 
 
 The main tool used in the proof of Theorem~\ref{taggedBC} is building structures with highly controlled automorphism groups.
 We recall the following terminology.
 
 \begin{Definition}  {\em  Given a structure $M$ and a subset $X\subseteq M^k$, we call $X$ {\em invariant} if it is preserved setwise under all automorphisms $\sigma\in Aut(M)$.
 }
 \end{Definition}
 
 It follows as an easy consequence of Scott's Isomorphism Theorem that if $M$ is any countable structure in a countable language $\L$ and if $X\subseteq M^k$ is any subset, then $X$ is invariant if and only
 if $X$ is definable by an infinitary  formula of $\L_{\omega_1,\omega}$.  In what follows, we will use this fact without comment.  In many cases, we will prove that some set is infinitarily definable and
 conclude that it is invariant; this is immediate, and does not use countability of $M$.

%
%

Fix any countable ring $R$, not necessarily commutative.  Throughout this section, we only consider left $R$-modules. 
Let $\L_R:=\{+,0,\cdot_r\}_{r\in R}$ be the usual language for (left) $R$-modules. 

\begin{Definition}  {\em  
A {\em tagged (left) $R$-module} is an expansion  $M=(V,+,0,\cdot_r,U_i)_{r\in R,i\in I}$ of a left $R$-module $(V,+,0 \cdot_r)_{r\in R}$
by a countably infinite set of unary predicates $U_i$, where the interpretation $U_i^M$ of each $U_i$ is a left $R$-submodule of $V$.
}
\end{Definition}

This notion is not new; for instance in  \cite{GS},  G\"obel and Shelah denote the class of countable tagged (left) $R$-modules as $R_\omega\mbox{-mod}$.


\begin{Remark}  \label{shift}
	{\em  Note that in the definition of tagged $R$-modules, we did not specify in advance what symbols to use as the new unary predicates. This has the following notational advantage: if $M$ is a tagged $R$-module,
		then so is any expansion $M^*=(M,W_i)_{i \in I}$ of $M$ by any countable sequence of submodules. This convention is inessential; we could alternatively just shuffle the named submodules of $M$ to make room.
	}
\end{Remark}

The goal of this section is to prove Theorem~\ref{taggedBC}, that the class of tagged $R$-modules is Borel complete for every countable ring $R$.
We accomplish this goal in the following steps. First, we isolate a suitable family $\mathbf{K}$ of tagged $R$-modules which satisfies disjoint amalgamation. Using a general theorem from the appendix, we construct a Fraisse-like limit $M$ of $\mathbf{K}$ whose automorphism group is very rich.
Next, we use this $M$ to construct a tagged $R$-module $N$ which serves as an engine for coding. Specifically, we will obtain a Borel operation assigning, to each graph $G$ on $\omega$, an $R$-submodule $U_G$ of $N$. This is arranged so that $(\omega, G) \mapsto (N, U_G)$ is the desired reduction.


\begin{quotation}
\noindent {\bf 
Throughout the whole of this section, fix a countable ring $R$ (possibly non-commutative) and also fix a maximal left ideal $I\subseteq R$.}
\end{quotation}

We will work with structures in three languages.


\begin{itemize}  \item  $\L=\{+,0\}\cup\{\cdot_r:r\in R\}$ is the language of left $R$-modules;
\item  $\L^+$ is the language of tagged left $R$-modules, i.e., $L$ adjoined with countably many unary predicate symbols $\{V_n:n\in\omega\}$; and
\item  $\L^+_X$ is $\L^+$ with a unary predicate symbol $X$ adjoined.
\end{itemize}


\subsection{$I$-Bases of $R$-modules}


In this first subsection, we deal only with (left) $R$-modules. Recall that if $V$ is an $R$-module and $X \subseteq V$ then $RX$ is the $R$-submodule of $V$ generated by $X$.

\begin{Definition}
{\em Suppose $V$ is an $R$-module. Then say that $X \subseteq V$ is {\em $I$-independent} if for all distinct $\{x_i: i < n\}\subseteq X$ and for all $\{r_i: i < n\}\subseteq R$, $\sum_{i < n} r_i x_i = 0$ if and only if each $r_i \in I$.   [Note the `if and only if'.  Being $I$-independent also conveys positive information.]
Say that $X$ is an {\em $I$-basis for $V$} if it is $I$-independent, and also $RX = V$.

Let $\Delta(V)$ denote the set of all $a \in V$ with $\mbox{Ann}(a) = I$.  As a set $\Delta(V)$ does not have pleasant closure properties, e.g., $0\not\in \Delta(V)$.
}
\end{Definition}
The following facts are easily established.

\begin{Fact}  \label{ringfacts}
Suppose $V$ is an $R$-module. Then:
\begin{enumerate}
\item  For every $a\in \Delta(V)$,  $Ra$ is a simple $R$-module.
\item   For any $a,b\in \Delta(V)$, either $Ra=Rb$ or $Ra\cap Rb=0$.
\item If $X$ is $I$-independent then $X\subseteq \Delta(V)$ and  $Rx\cap Ry=0$ for distinct $x,y\in X$.
\item  If $X$ is an $I$-basis then $V$ is a semisimple (left) $R$-module.  In fact, $V\cong\bigoplus_X R/I$ via the map $f(\sum r_ix_i)(x_i)=(r_i+I)$.
\item  If $X$ is an $I$-basis of $V$ and $X'$ is an $I$-basis of $V'$ and $V \cap V' = 0$, then $X \cup X'$ is an $I$-basis of $V \oplus V'$.
\end{enumerate}
\end{Fact}

\begin{Definition}  {\em  Suppose $V$ is an $R$-module.  For any $Y\subseteq V$, let 
$$R*Y:=\{ry:r\in R, y\in Y, \hbox{and $ry\neq 0\}$}.$$ 
To contrast, note that for any $Y \subseteq V$, $RY$ is by definition the $R$-submodule of $V$ generated by $Y$, i.e. all finite sums from $R * Y$.

For $u,v\in V$, put $u{\sim}v$ if and only if $Ru=Rv$.
}
\end{Definition}


Clearly, $\sim$ is an invariant equivalence relation on $V$, and if $X$ is an $I$-basis then the elements of $X$ are pairwise $\sim$-inequivalent. Moreover, for any $u,v\in R*X$,
$u{\sim} v$ iff $(\exists r\in R) ru=v$ iff $(\exists s\in R) u=sv$ (since $Ru$ and $Rv$ are simple); and $R*X$ is the $\sim$-saturation of $X$.


\begin{Lemma}  \label{algebra}  Suppose $V$ is an $R$-module, $X$ is an $I$-independent subset of $V$ and $Y \subseteq X$. Then $(R* X) \cap RY = R*Y$.
\end{Lemma}


\begin{proof} Clearly $R* Y \subseteq (R* X) \cap RY$. Conversely, suppose $rx \in (R*X) \cap RY$, where $x \in X$. It suffices to show $x \in Y$. Suppose it weren't. Since $rx \in RY$ we can write $rx = \sum_{i < n} r_i y_i$ where $(y_i:i < n)$ are distinct elements of $Y$. By $I$-independence, we have that $r \in I$, but this implies $rx = 0$, contradicting $rx \in R*X$. \qed
\end{proof}

\subsection{Tagged left $R$-modules with nearly invariant bases}


We now discuss tagged left $R$-modules, i.e., $\L^+$-structures $(V,V_n)_{n\in\omega}$ and their expansions by a unary predicate $X$.

\begin{Definition}  {\em  Let $\Phi_0$ be the $(\L^+_X)_{\omega_1,\omega}$-sentence asserting of its model $(V, X, V_n)_{n \in \omega}$ that $V$ is a left $R$-module, each $V_n$ is an $R$-submodule, and that $X$ is an $I$-basis of $V$.}
\end{Definition}

Thus, models of $\Phi_0$ are tagged left $R$-modules with named $I$-basis. Now our goal is to produce complicated tagged $R$-modules, and to do so we will eventually have to forget $X$, since $X$ is not a submodule. Ideally we would want to be able to recover $X$ from the rest of the structure, but this turns out to be too much to ask. So instead we just try to make $R*X$ recoverable from the rest of the structure; this is good enough in practice, because $R*X/\sim$ is in bijection with $X$.

The key trick is as follows. Let $\Phi_1$ be the $(\L_X^+)_{\omega_1,\omega}$-sentence asserting
$$\Phi_1:= \Phi_0\wedge   \hbox{``For all $a\in \Delta(V)$, $a\in R*X$ if and only if $a\not\in \bigcup\{V_n:n\in\omega\}$.''}$$


\begin{Lemma}  \label{RXinvariant}  Suppose that $(V,X, V_n)_{n\in\omega}\models \Phi_0$.
Then $(V,X,V_n)_{n\in\omega}\models \Phi_1$ if and only if $R*X=R*(\Delta(V)\setminus \bigcup\{V_n:n\in\omega\})$. In particular, this implies $R*X$ is invariant in $(V, V_n)_{n \in \omega}$.
\end{Lemma}

\begin{proof}  First, assume $(V,X,V_n)_{n\in\omega}\models \Phi_1$.  To see `$\subseteq$', choose any $x\in X$.  Then $x\in \Delta(V)$ and also $x \in R*x$.
Since $(V,X,V_n)_{n\in\omega}\models\Phi_1$, we get $x\not\in \bigcup\{V_n:n\in\omega\}$, so $x\in \Delta(V)\setminus\bigcup\{V_n:n\in\omega\}$.
Thus, $R*x\subseteq R*(\Delta(V)\setminus \bigcup\{V_n:n\in\omega\})$.  As this holds for all $x\in X$, 
$R*X\subseteq R*(\Delta(V)\setminus \bigcup\{V_n:n\in\omega\})$.
To see `$\supseteq$', choose any $a\in (\Delta(V)\setminus \bigcup\{V_n:n\in\omega\})$.
Since $(V,X,V_n)_{n\in\omega}\models\Phi_1$, $a\in R*X$.  Thus, $Ra\subseteq R*X$.  So  `$\supseteq$' follows.

Conversely, assume that $R*X=R*(\Delta(V)\setminus \bigcup\{V_n:n\in\omega\})$.
To see that $(V,X,V_n)_{n\in\omega}\models \Phi_1$, choose any $a\in \Delta(V)$.  First, assume $a\in R*X$.  By our assumption, there is $e\in (\Delta(V)\setminus \bigcup\{V_n:n\in\omega\})$ such that
$a\in R*e$.  As both $a,e\in \Delta(V)$, simplicity implies that $Ra=Re$.  Thus, if $a\in V_n$ for some $n$, then $e\in Ra\subseteq V_n$, which is a contradiction.  Thus,
$a\not\in\bigcup\{V_n:n\in\omega\}$.  Finally, suppose $a\not\in\bigcup\{V_n:n\in\omega\}$.  As we chose $a\in \Delta(V)$, our assumption implies $a\in R*X$. \qed
\end{proof}

\medskip


As an indexing device, fix a partition  $\omega=\bigsqcup\{J_k\in\omega\}$ with each $J_k$ infinite.  
Let $\K$ be the class of all $\L_X^+$-structures $(V, X, V_n)_{n \in \omega}$ such that:
\begin{itemize}
	\item $(V, X, V_n)_{n \in \omega} \models \Phi_1$;
\item There is some $k_*$ such that $V_n=0$ for every $n\in \bigcup_{k \geq k_*} J_k$;
\item $X^K$ is finite.
\end{itemize}

\begin{Lemma}  \label{modsuit} The class $\K$ is suitable (see Definition~\ref{GrowingDef}).  In fact, $\K$ is strongly suitable.
\end{Lemma}
\begin{proof}  The only interesting part is verifying disjoint amalgamation.
Choose $M,M_0,M_1\in \K$ with $M$ an $\L_X^+$-substructure of both $M_0$ and $M_1$ and $M_0\cap M_1=M$.
 Write $X_i = (X)^{M_i}$. Let $V'$ be the $R$-module generated  by $X':=X_0 \cup X_1$, subject to the relations $\sum r_i x_i =0 $ precisely when each $r_i \in I$.  Thus, each $M_i$ is a submodule of $V'$. 
Let  $N' = (V',X', V'_n)_{n\in\omega}$, where $V'_n = V_n^{M_0} + V_n^{M_1}$. It is readily checked that $N'\models \Phi_0$, $X'$ is finite, and there is some least $k^*$ so that $V'_n=0$  for all
$n\in \bigcup_{k \geq k_*} J_k$.


\medskip
\noindent{\bf Claim.}  For $i=0,1$, if $a\in M_i\cap V_n'$ for some $n$, then $a\in V^{M_i}_n$.


\begin{proof}  By symmetry, take $i=0$.  Say $a\in M_0$ and $a=b_0+b_1$ with $b_0\in V_n^{M_0}$ and $b_1\in V_n^{M_1}$.  
Since $b_1=a-b_0$, $b_1\in M_0$. Since the structures $M_0$ and $M_1$ agree on their intersection, this implies $b_1 \in V_n^{M_0}$. Thus $a =b_0+b_1\in V_n^{M_0}$. \qed
\end{proof}

It follows immediately from the Claim that both $M_0$ and $M_1$ are $\L_X^+$-substructures of $N'$.  Unfortunately, 
 $N'$ need not be a model of $\Phi_1$ (nor $\Phi_2$) but a slight alteration of $N'$ will be.
 The problem is that there may be `extra' elements $e\in \Delta(V')\setminus\bigcup\{V_n':n\in\omega\})$ that are not in $R*X'$.
 Let $E$ be the set of these problematic elements.  As $E$ is countable and $J_{k^*}$ is countably infinite, 
associate  distinct $n(e)\in J_{k^*}$ to every $e\in E$.  


Let $N$ be the $\L_X^+$-structure with the same $(V',X')$ as $N'$, where, for each $e\in E$, $V_{n(e)}^N$ is interpreted as the (simple) submodule $Re$,
and
$V_n^N$ interpreted as $V_n'$ when $n\neq n(e)$ for any $e\in E$.

As $N$ and $N'$ have the same underlying $R$-module and same interpretation of $X$, $N\models\Phi_0$.  To see that $N\models \Phi_1$, choose any $a\in \Delta(V')$.
First, if $a\in (V'\setminus\bigcup\{V_n^N:n\in\omega\})$, then $a\not\in E$, lest $a\in V_{n(a)}^N$.   So, by definition of $E$, $a\in R*X'$.
Conversely, assume $a\in R*X'$.  Then, for any $e\in E$ it follows from simplicity that $Ra\cap Re=\emptyset$, hence $a\not\in V_{n(e)}^N$ for any $e\in E$.
Thus, it suffices to show $a\not\in V_m'$ for any $m \in \bigcup_{k < k_*} J_k$.  By way of contradiction, if $a=rx$ for some $x\in M_i$ for some $i < 2$, then it follows from the Claim
that $a\in V_m^{M_i}$, contradicting $M_0\models\Phi_1$.  Thus, $N\models \Phi_1$.


But also, by incrementing $k^*$ by 1, we see that $N\in \K$.

%
%

It remains to show that both $M_0, M_1$ are $\L_X^+$-substructures of $N$.  As we know already that both are substructures of $N'$, 
the only non-trivial point to check is that $V_{n(e)}^N\cap M_i=0$ for each $e\in E$ and $i<2$.  
Say $re\in V_{n(e)}^N$ is non-zero.    If $re$ were in (say) $M_0$, then as $Re$ is a simple module, we would also have $e\in M_0$.
But, as $e\in E$, $e\not\in R*X'$; hence $e \not \in R * X_0$.   Since $M_0\models\Phi_1$, we would have $e\in V_m^{M_0}\subseteq V_m'$ for some $m$, contradicting $e\in E$. \qed
\end{proof}

\begin{Proposition}  \label{M}  There is a countable tagged left $R$-module $M= (V, V_n)_{n \in \omega}$ with $I$-basis $X$ and an equivalence relation $E$ on $X$ satisfying: 
\begin{enumerate}
\item  $(V, X, V_n)_{n \in \omega} \models \Phi_1$, and so in particular $R*X$ is invariant in $(V, V_n)_{n \in \omega}$;
\item  $X/E$ is infinite; and
\item  Every $h\in Sym(X/E)$ lifts to an $\L^+_X$-automorphism of $(M, X)$, i.e., there is $\sigma \in \mbox{Aut}(M, X)$ such that for every $x\in X$,
$\sigma(x)/E=h(x/E)$.
\end{enumerate}
\end{Proposition}  


\begin{proof}  Apply Theorem~\ref{GeneralSuitable} to get
an $\L_X$-structure $M$ and equivalence relation $E$ on $X^M$ as there.
Items (2) and (3) are immediate; item (1) uses that models of $\Phi_1$ are closed under unions of chains. \qed
\end{proof}



Fix $M$ as in the previous proposition.

\begin{Definition}\label{E}
	{\em  Let $\E$ be the equivalence relation on $R*X$ defined via $rx \E sy$ if and only if $x E y$, for all $x, y \in X$ and $r, s \in R \backslash I$. So the classes of $\E$ are just the $\sim$-saturations of the classes of $E$, which gives a bijection between the classes of $\E$ and the classes of $E$.}
\end{Definition}

It follows that every $h\in Sym((R*X)/\E)$ lifts to an automorphism of $(M, X)$.


\subsection{Constructing the engine $N$}

We continue; so we have fixed the countable ring $R$, the maximal left ideal $I$, the model $(M, X, V_n)_{n \in \omega}$ and equivalence relations $E, \E$, as in Proposition ~\ref{M} and Definition~\ref{E}.

\begin{Proposition}  \label{N}  There is a countable tagged left $R$-module 
$N = (V', V'_n)_{n \in \omega}$ with distinguished
$I$-basis $X_0 \sqcup X_1$; equivalence relations $\E_0,\E_1$ on $R*X_0,R*X_1$, respectively;
and  a bijection $k:[(R*X_0)/{\sim}]^2\rightarrow R*X_1/\E_1$
 satisfying:
\begin{enumerate}
\item  $(R*X_0)/\E_0$ is infinite; and
\item  The following sets are invariant in $N$:
 \begin{enumerate}
 \item  $R* X_0$ and $R*X_1$;
 \item $\E_0$ and $\E_1$; and
 \item  the set $K \subseteq R*X_0 \times R* X_0 \times R* X_1$ of all triples $(rx, sy, tz)$ satisfying that $k(rx/{\sim}, sy/{\sim}) = tz/\mathcal{E}_1$. (We would like to say $k$ is invariant but it isn't of the right form.)
 \end{enumerate}
 \item  Every $h\in Sym((R*X_0)/\E_0)$ lifts to an automorphism of $(N, X_0, X_1)$ (i.e. an automorphism of $N$ fixing $X_0$ and $X_1$ setwise).
 \end{enumerate}
 \end{Proposition}



\begin{proof}  We begin by defining a doubling $N^-$ of $M$, and then we name some additional submodules. To begin, let the underlying $R$-module $V'$ be $V \times V$, where recall $V$ is the underlying module of $M$. Let $X_0 = X \times 0$ and let $X_1 = 0 \times X$, so $X_0 \cup X_1$ is an $I$-basis of $V'$. Let $N^-$ be $V'$ equipped with the named submodules $V \times 0, 0 \times V, V_n \times 0$ and $0 \times V_n$, for $n < \omega$.  Note that the projection maps $\pi_0, \pi_1: V' \to V$ are invariant in $N^-$. Further, since $R*X$ is invariant in $M$, we get that $R* X_0$ and $R* X_1$ are both invariant in $N^-$. Finally, note that given $\sigma, \tau \in \mbox{Aut}(M)$, $\sigma \times \tau$ is an automorphism of $N^-$, and in fact every automorphism of $N^-$ is of this form. Let $E_i, \E_i$ be the copies of $E, \E$.
	
As both $X_0$  and $R*X_1/\E_1$ are countably infinite by Corollary~\ref{E},  choose any bijection $k_0:[X_0]^2\rightarrow R*X_1/\E_1$. 
Recall that since each $Rx$ is simple, there is a natural bijection between $X_0$ and $(R*X_0)/{\sim}$, namely $x\mapsto x/{\sim}$.  Composing $k_0$ with this bijection yields a bijection
 $$k:[(R*X_0)/{\sim}]^2\rightarrow (R*X_1)/\E_1$$


Additionally, let $$T:=\{x+y+z:\{x,y\}\in [X_0]^2, z\in k_0(\{x,y\})\}$$
and let $$Q:=\{z\in X_1:\hbox{there exist $\{x,y\}\in [X_0]^2$ with $x+y+z\in T$ and $\E_0(x,y)$}\}.$$

Let $N:=(N^-,RQ, RT)$, an expansion of $N^-$ by two submodules.  We verify that $N, X_0, X_1, \E_0, \E_1$ satisfy the requirements of the Proposition.  We already noted that $N$ is a countable tagged $R$-module with $I$-basis $X_0 \cup X_1$, and that $R* X_0$ and $R* X_1$ are invariant.

We first establish the following claim.

\medskip





\medskip
\noindent{\bf Claim.}  $R*T$ is invariant in $N$.
\begin{proof} It is enough to show that $w \in R*T$ if and only if $w \in RT$ and $\pi_1(w) \in R*X_1$, since $RT, \pi_1$ and $R*X_1$ are all invariant in $N$. The forward direction is clear. For the reverse, suppose $w \in RT$ has $\pi_1(w) \in R*X_1$. Write $w=\sum_{i<n} r_it_i$ with each $r_i\in R$ and each $t_i\in T$.
By combining repeats, we may assume $\{t_i:i<n\}$ is without repetition, and by eliminating zero terms, we may assume each $r_i\in R\setminus I$.
Write each $t_i$ as $(x_i+y_i)+z_i$ and note that $\pi_1(t_i)=z_i$. Further, note that $(z_i: i < n)$ are all distinct, as if $z_i = z_j$, then also $t_i = t_j$, contrary to arrangement.


We have $\pi_1(w)=\sum_{i<n} r_iz_i$.  But also, $\pi_1(w)\in R*X_1$ implies that $\sum_{i<n} r_iz_i=r^*z^*$ for some $r^*\in R\setminus I$ and some $z^*\in X_1$.
By $I$-independence of $X_0 \cup X_1$, we conclude that $n=1$, as desired.  \qed
\end{proof}

It now follows that $K$ is invariant, since it is the set of all tuples $(a, b, c) \in R* X_0 \times R*X_0 \times  R*X_1$ such that there exist $r, s, t$ with $ra + sb + tc \in R * T$ (using simplicity of $Ra, Rb$ and $Rc$).

We next show that $\E_1$ and $\E_0$ are invariant in $N$. Towards this, we make the following definition, which will also be used in the next section.

\begin{Definition}\label{Kdef}
{\em Suppose $A \subseteq R * X_1$. Then let $K^{-1}(A) = \{(b_0, b_1) \in R* X_0 \times R * X_0: (b_0, b_1, a) \in K\}$. }
\end{Definition}

\medskip
\noindent{\bf Claim.} $\E_1$ and $\E_0$ are invariant in $N$.

\begin{proof}
	Note that if $a \in R* X_1$, say $a \in k_0(\{x, y\})$, then $K^{-1}(a) = (x/\sim) \times (y/\sim)$, which depends precisely on $a/\mathcal{E}_1$. Hence, given $a_0, a_1 \in R * X_1$, we have $a_0 \E_1 a_1$ if and only if $K^{-1}(\{a_0\}) = K^{-1}(\{a_1\})$, and so $\E_1$ is invariant in $N$. 
	
	Also, note that $R * Q$ is invariant in $N$, since by Lemma~\ref{algebra} applied to $Q \subseteq X_1$ it is equal to $RQ \cap R*X_1$. Hence $K^{-1}(R*Q)$ is invariant, but $K^{-1}(R*Q) = \E_0$, so $\E_0$ is invariant.  \qed
\end{proof}

Finally, it suffices to show that every $h\in Sym((R*X_0)/\E_0)$ lifts to an automorphism of $N$.  Given such an $h$, by Corollary~\ref{E} there is an automorphism
$\sigma_0\in Aut(M \times 0, X_0)$ lifting $h$. Note that $\sigma_0$ must also preserve $\E_0$.


 Let $h_1\in Sym(X_1/E_1)$ be the unique permutation satisfying
$$h_1(k_0(\{x,y\}))=k_0(\{\sigma_0(x),\sigma_0(y)\})$$
for all $\{x,y\}\in [X_0]^2$. By Proposition~\ref{M}, we can find an automorphism $\sigma_1\in Aut(0 \times M, X_1)$ that is a lifting of $h_1$.
Then we can define $\sigma\in \mbox{Aut}(N^-)$ by $\sigma(a+b)=\sigma_0(a)+\sigma_1(b)$ for all $a\in M \times 0$, $b\in 0 \times M$.
We know that $\sigma$ preserves $X_0, X_1$ and $\E_0$ setwise, and by arrangement $\sigma$ commutes with $k_0$. Since $RT$ and $RQ$ can be defined from $X_0, X_1, k_0$ and $\E_0$, $\sigma$ also preserves $RT$ and $RQ$ setwise; hence $\sigma\in Aut(N, X_0, X_1)$.   \qed
\end{proof}

\subsection{Tagged $R$-Modules Are Borel complete for every countable ring $R$}

Using the engine $N$ given in Proposition~\ref{N}, we prove Theorem~\ref{taggedBC}, that the theory of tagged left $R$-modules is Borel complete for any countable ring $R$.

\begin{proof}  Given $R$, fix a maximal left ideal $I\subseteq R$ and choose $N, X_0, X_1, \E_0, \E_1, K$ as in Proposition~\ref{N}.  Note that by Remark~\ref{shift},
any expansion $(N,U)$ by a submodule is still a tagged left $R$-module.
By a countable graph we mean an irreflexive, symmetric binary relation.  Since $R*X_0 / \E_0$ is countably infinite,
we may take this set to be the universe of our graphs.
 
 Now, for every such graph $(R*X_0/\E_0, G)$, let $Q_G$ be the set of all $z \in X_1$ such that there are $a,b \in R*X_0 $ with $(a, b, z) \in K$ and with $(a/\E_0, b/\E_0) \in G$. Let  $U_G=RQ_G$, the smallest $R$-submodule containing $Q_G$.  
 
 The map $G\mapsto (N,U_G)$ is clearly Borel, so we must show it preserves isomorphism in both directions.   
 
 First, note that by Lemma~\ref{algebra}, $R*Q_G$ is invariant in $(N, U_G)$, as it is equal to $U_G \cap X_1$. Hence $G$ can be recovered from $(N, U_G)$, since in the notation of Definition~\ref{Kdef}, $K^{-1}(R*Q_G)$ is the set of all $(a, b) \in R*X_0 \times R*X_0$ with $(a/\E_0, b/\E_0) \in G$, and $\E_0$ is invariant in $N$. This shows that if $(N, U_G) \cong (N, U_{G'})$ then also $G \cong G'$.
 
Conversely, suppose $h:G\rightarrow G'$ is a graph isomorphism. Then $h\in Sym(R*X_0/\E_0)$ is a permutation.  Let $\sigma\in Aut(N, X_0, X_1)$ be a lifting of $h$. An inspection of the definition of $Q_G$ shows that $\sigma$ takes $Q_G$ to $Q_{G'}$, which shows that $(N, U_G) \cong (N, U_{G'})$.   \qed
\end{proof}

 \section{The main Borel reductions}  \label{transfer}

 This section gives a number of Borel reductions that are used in Section~\ref{TWO} to prove Theorem~\ref{big}.

 \subsection{Quotient rings and free-like tagged $R$-modules}
 
  Although incredibly simple, our first reduction will be used many times.
 
 \begin{Lemma}  \label{R/I}  For any countable ring $R$ and any 2-sided ideal $I\subseteq R$, there is a Borel reduction from 
 the class of countable left $R/I$-modules to the class of countable left $R$-modules.  Thus, if the theory of left $R/I$-modules is Borel complete,
 then so is the theory of left $R$-modules.
 \end{Lemma}
 
 \begin{proof}  Any $R/I$-module $M$ can be naturally expanded to an $R$-module by positing $ra:=(r+I)a$ for every $a\in M$.  Thus, the identity map
 gives a Borel reduction from left $R/I$-modules to left $R$-modules.  The second statement is immediate because of the transitivity of Borel reductions.  \qed
 \end{proof}


\begin{Definition}  {\em  Let $R$ be any countable ring.  A tagged left $R$-module $(M,M_n)_{n\in\omega}$ is {\em free-like} if the following  conditions hold:
\begin{enumerate}
\item  The universe of $M$ is isomorphic to the free left $R$-module $\bigoplus_\omega R$;
\item For every $n\in\omega$, both  $M_n$ and the quotient module $M/M_n$ are isomorphic to $\bigoplus_\omega R$.
\end{enumerate}
}
\end{Definition}  

\begin{Remark}\label{freelikeImpliesPure}
If $(M,M_n)_{n \in \omega}$ is free-like, then each $M_n$ is a direct summand of $M$ (since $M/M_n \cong \bigoplus_\omega R$ is projective), and so each $M_n$ is pure in $M$.
\end{Remark}

\begin{Proposition}   \label{freelikeBC} For any countable ring $R$, the class of free-like, tagged left $R$-modules is Borel complete.
\end{Proposition}

\begin{proof}  Fix a countable ring $R$.  For the ease of notation and understanding we break this reduction into two parts.
We first show that the class of tagged left $R$-modules is Borel reducible to the class of tagged left $R$-modules whose universe is isomorphic to  $\bigoplus_\omega R$.
To see this, suppose we are given a countable, tagged left $R$-module $(M,M_n)_{n\in\omega}$.  
Let $M^*:=\bigoplus_M R$, i.e., the free left $R$-module with basis $\{\ee_a:a\in M\}$.  
Define an $R$-module homomorphism
$$\delta:M^*\rightarrow M$$
via $\delta(\sum_{a\in M} r_a \ee_a)=\sum_{a\in M} r_a a$.  [Since $M^*$ is a direct sum, all but finitely many $r_a=0$, so this definition makes sense.]
Let $$f(M):=(M^*,\ker(\delta),\delta^{-1}(M_n))_{n\in\omega}$$
Visibly, $f(M)$ is a tagged left $R$-module whose universe is isomorphic to $\bigoplus_\omega R$.  It is easily checked that the function $f$ is Borel and that 
$f(M)\cong f(M')$ whenever $M$ and $M'$ are isomorphic tagged $R$-modules.
For the converse, it suffices to show we can recover $M/{\cong}$ from $f(M)/{\cong}$. But this is clear, since $(M, M_n)_{n \in \omega} \cong (M_*/\ker(\delta), \delta^{-1}(M_n)/\ker(\delta))_{n \in \omega}$.

This works, except in the case when $M^*$ is finite-dimensional (i.e. $M$ is finite); in that case, replace $M^*$ by $\bigoplus_{M \times \omega} R$.

For the second step, we note the following general fact.

\begin{Lemma}   \label{help} Suppose $M=M_0\oplus M_1$ is any direct sum of left $R$-modules, and let $h:M_0\rightarrow M_1$ be any $R$-module homomorphism.
Let $N=\{a-h(a):a\in M_0\}$.  
Then $N\cong M_0$ and $M=N\oplus M_1$, hence $M/N\cong M_1$.  In particular, if both $M_0$ and $M_1$ are free, then so are
$N$ and $M/N$.
\end{Lemma}

\begin{proof}  There is a natural surjective homomorphism $f:M_0\rightarrow N$ given by $f(a)=a-h(a)$.  However, if $b\in \ker(f)$, then $b\in M_0$ and $b=h(b)$.
As $M=M_0\oplus M_1$, $b=0$.  Thus, $f$ is an $R$-module isomorphism.   This isomorphism extends naturally to an automorphism $f^*:M\rightarrow M$
as $f^*(a+b)=a-f(a)+b$. Since $f^*$ is the identity on $M_1$ and takes $M_0$ to $N$, we get that $M=N\oplus M_1$.  It follows immediately that $M/N\cong M_1$.  \qed
\end{proof}


We now give a Borel reduction from tagged left $R$-modules with universe $\bigoplus_\omega R$ to free-like tagged left $R$-modules.  Fix a basis $\{\ee_i:i\in\omega\}$ for $\bigoplus_\omega R$.
Given $M=(\bigoplus_\omega R,M_n)_{n\in\omega}$, 
let $$J:=\{\ee_i:i\in\omega\}\cup \bigcup_{n\in\omega} \{\dd_{n,\aa}:\aa\in M_n\}$$
[if some $M_n$ is finite, take $\{d_{n,\aa,k}:\aa\in M_n,k\in\omega\}$ in place of  $\{\dd_{n,\aa}:\aa\in M_n\}$.]
Let $\bigoplus_J R$ be the free $R$-module with basis $J$; this will be the underlying module of our output $f(M)$.

Let $U_*$ be the submodule of $\bigoplus_J R$ generated by $(\mathbf{e}_i: i \in \omega\}$. Note that this is the domain of $M$, so each $M_n$ is a submodule of $U_*$.
Also, for each $n\in\omega$, let $U_n$ denote the $R$-submodule of $\bigoplus_J R$ generated by $\{d_{n,\aa}:\aa\in M_n\}$.

Let $\epsilon_n:U_n\rightarrow U_*$ be the $R$-module homomorphism generated by the map $\dd_{n,\aa}\mapsto \aa$.  
Note that the image of $\epsilon_n$ is $M_n$. Finally, let $V_n:=\{b-\epsilon_n(b):b\in U_n\}$ and let $N(M)$ be the tagged left $R$-module
$$(\bigoplus_J R, U_*, U_n,V_n)_{n\in\omega}$$
Clearly, $N(M)$ is a tagged left $R$-module whose universe is isomorphic to $\bigoplus_\omega R$, with $U_*$, $N(M)/U_*$, $U_n$ and $N(M)/U_n$ all isomorphic to $\bigoplus_\omega R$.
As well, it  follows from Lemma~\ref{help} that each $V_n$ and each $N(M)/V_n$ is isomorphic to $\bigoplus_\omega R$.
Thus, $N(M)$ is a free-like tagged left $R$-module.

It is easily checked that the mapping $M\mapsto N(M)$ is Borel and that $M\cong M'$ implies $N(M)\cong N(M')$.  For the reverse direction, it suffices to show we can recover $M/{\cong}$ from $N(M)/{\cong}$. In fact, we recover $M$ on the nose from $N(M)$. To see this, note that for a given $n$, and for all $(b, c) \in N(M)$, we have that $(b, c) \in \mbox{graph}(\epsilon_n)$ if and only if $b \in U_n$ and $b - c \in V_n$. Hence we can recover $M_n$ as the image of $\epsilon_n$.  \qed
\end{proof}

\subsection{Rings with a central, non-zero divisor, non unit element}


The goal of this subsection is to prove the first of our three main Borel reductions.

\begin{Theorem}  \label{A}  If $R$ is any countable ring with a central element $r\in R$ that is a neither a zero divisor nor a unit, then the theory of left $R$ modules is Borel complete.
\end{Theorem}

We do this by finding a countably infinite, sufficiently independent subset $\Gamma$ of the natural inverse limit ring $\Rhat$; we then use $\Gamma$ to code a free-like
 tagged left $R$-module into a single left $R$-module.  The construction is analogous to the $p$-adic construction, but here we do not put any primeness condition on $(r)$.

 \begin{Definition}  {\em  For a ring $R$, an element $r\in R$ is {\em central} if $rs=sr$ for every $s\in R$.  
 For central $r$,  we say $r$ {\em is a zero divisor} if some non-zero $s\in R$ satisfies $rs=0$,  and $r$ is a {\em unit} if  $rs=1$ for some $s\in R$.
 As we only use these adjectives  on central elements, we do not need to distinguish between left and right.
 }
 \end{Definition}
 
 Now, for the whole of this subsection, fix countable ring $R$ and a central element $r\in R$ that is neither a zero divisor nor a unit.  The following facts are easily verified:
 
 \begin{Fact}  \label{Rfacts}
 \begin{enumerate}
 \item  For all $m\ge 1$, $(r^m):=\{s\in R: s=tr^m$ for some $t\in R\}$ is a two-sided ideal.
 \item For all $m$, $r^m$ is not a zero divisor. 
 \item Generally (in any ring $R$), if $a \in R$ is central and not a zero divisor, then $ax = ay$ implies $x = y$ for all $x, y \in R$.
 \item  $1\not\in (r)$ and, for $m\ge 1$, $r^m\not\in (r^{m+1})$.
 \end{enumerate}
 \end{Fact}
 
 It follows that $\{(r^m):m\ge 1\}$ is a strictly decreasing sequence of two-sided ideals, hence $J:=\cap_{m\ge 1} (r^m)$ is also a two-sided ideal, hence $R/J$ is also a countable ring.
 Let $\rho:R\rightarrow R/J$ denote the canonical surjection.  
 
 \begin{Lemma}  The element $\rho(r)$ is central, and is neither a zero divisor nor a unit in $R/J$.  
 \end{Lemma}
\begin{proof} That $\rho(r)$ is central in $R/J$ is easy. 
 To see that $\rho(r)$ is not a zero divisor, choose $s\in R$ such that $\rho(r)\rho(s)=0$ in $R/J$.  This implies $rs\in J$, hence $rs\in (r^{m+1})$ for every $m\ge 1$.
 For each $m\ge 1$, choose $t_m\in R$ so that $rs=t_mr^{m+1}$.  Since $r$ is not a zero divisor, $s=t_m r^m$, so $s\in (r^m)$.  As $m$ is arbitrary, $s\in J$, so $\rho(s)=0$ in $R/J$.
 Finally, to see that $\rho(r)$ is not a unit, first note that for any $s\in R$, $rs-1 \not \in (r^1)$.  [If it were, say $rs-1=rt$ for some $t\in R$, then we would have $r(s-t)=1$, contradicting
 $r$ a non-unit.] In particular, $rs-1\not\in J$, hence $\rho(r)\rho(s)\neq 1$ in $R/J$.  \qed
\end{proof}
 
 In light of Lemma~\ref{R/I}, {\bf we may now further assume that $J=\bigcap_{m\ge 1} (r^m)=\{0\}$.}

%
%
%
 
 \medskip
Associated to the descending sequence $\<(r^m):m\ge 1\>$ of (two sided) ideals we have a commuting family of surjections $\{\pi_{n,m}:R/(r^n)\rightarrow R/(r^m):1\le m\le n\}$.
Let 
$$\Rhat:=\varprojlim R/(r^m)$$ denote the inverse limit of these rings with respect to these maps.

Thus, elements $\delta\in \Rhat$ are sequences $\<\delta(m):m\ge 1\>$ with each $\delta(m)\in R/(r^m)$ and $\pi_{n,m}(\delta(n))=\delta(m)$ whenever $1\le m\le n$.
Note that an element of $R/(r^m)$ is formally a coset of $(r^m)$, and thus is a subset of $R$. Abusing notation somewhat, we identify an element $c\in R$ with the sequence $\<c+(r^m):m\ge 1\>$ in $\Rhat$; as we are assuming $\bigcap_m (r^m) =0$, this identification is faithful. We thus construe $R$ as a subring of $\Rhat$.
As notation, for each $c\in R$, let $$c\Rhat:=\{c\delta:\delta\in \Rhat\}$$

\begin{Lemma} \label{charideal} For any $n\ge 1$, $r^n\Rhat=\{\gamma\in\Rhat:\gamma(n)=(r^n)\}$.
\end{Lemma}
(To parse the lemma statement, note that $(r^n)$ is the $0$-element of $R/(r^n)$, so saying $\gamma(n) = (r^n)$ is the same as saying $\gamma(n) =0 $.)
\begin{proof}  Fix any $n\ge 1$.  That $(r^n\delta)(n)=(r^n)$ is obvious.  Conversely, choose any $\gamma\in \Rhat$ with $\gamma(n)= (r^n)$ and we will construct some $\delta\in\Rhat$ with
$r^n\delta=\gamma$.  For all $m\ge n+1$, choose $s_m\in \gamma(m)$ (so $s_m\in R$).  Since $\pi_{n,m}(\gamma(m))=\gamma(n)$ whenever $m\ge n$, we have each $s_m\in (r^n)$,
so choose $t_{m-n}\in R$ such that $s_m=r^nt_{m-n}$.

\medskip
\noindent
{\bf Claim}  For $m'\ge m\ge n+1$, $t_{m'-n}\in t_{m-n}+ (r^{m-n})$.

\bp  Since $\pi_{m',m}(\gamma(m'))=\gamma(m)$, $s_{m'}\in \gamma(m)$, so $s_{m'}-s_m\in (r^m)$.  Thus,
$$r^n(t_{m'-n}-t_{m-n})\in (r^m)$$
so $(t_{m'-n}-t_{m-n})\in (r^{m-n})$ since $r^n$ is not a $0$-divisor.  \qed
\end{proof}

By the Claim, $\delta:=\<t_{m-n}+(r^{m-n}):m\ge n+1\>$ is an element of $\Rhat$, and it is easily checked that $\gamma=r^n\delta$, hence $\gamma\in r^n\Rhat$.


%

\begin{Lemma}  In $\Rhat$:
\begin{enumerate}
\item  The element $r$ remains central, a non-zero divisor, and a non-unit; and
\item  $(r^m\Rhat:m\ge 1)$ is a strictly descending sequence of 2-sided ideals with $\bigcap_{m\ge 1} r^m\Rhat=0$.
\end{enumerate}

\end{Lemma}

\begin{proof} First, we check that $r$ is central in $\Rhat$. Given $\delta \in \Rhat$, we have that for every $m$, $(r \delta)(m) = (r)(m) \delta(m) = \delta(m) r(m) = (\delta r)(m)$, using that $r(m) = r + (r^m)$ is central in $R/(r^m)$.

From this it follows that each $r^m$ is central in $\Rhat$, and hence each $r^m \Rhat$ is a two-sided ideal of $\Rhat$. By Lemma~\ref{charideal} and definition of $\Rhat$, we see that $\bigcap_{m \geq 1} r^m \Rhat = 0$. So item (2) has been established, and it remains to check that $r$ is neither a zero divisor nor a unit.

$r$ cannot be a unit, because otherwise we would have $1 \in r \Rhat$, but this contradicts Lemma~\ref{charideal}, since $1(1) = 1 + (r) \not= (r)$.

Finally, suppose $\delta \in \hat{R}$ satisfies $r \delta = 0$, and suppose towards a contradiction $\delta \not= 0$. Then we can find some $m$ with $\delta(m) \not= (r^m)$. Choose $s_m \in \delta(m)$ and $s_{m+1} \in \delta(m+1)$. Then $s_{m+1} - s_m \in (r^m)$ and $s_m \not \in (r^m)$, so $s_{m+1} \not \in (r^m)$. Hence $r s_{m+1} \not \in (r^{m+1})$. Hence $(r \delta)(m+1) \not= (r^{m+1})$, contradicting $r \delta = 0$.
\qed

\end{proof}

We identify desirable subsets of $R$ and $\Rhat$, respectively.
For each finite $s\subseteq \omega$, let 
$$\eps_s=\sum_{i\in s} r^i$$ and let
$S=\{\eps_s:s\subseteq\omega$ finite$\}$.  
Let $\Gamma=\varprojlim S$, the inverse limit of $S$ with the natural maps.  That is, 
$$\Gamma=\{\gamma\in\Rhat: \gamma(m)\cap S\neq\emptyset\ \hbox{for all $m\ge 1$}\}$$


A more explicit way of viewing $\Gamma$ is as follows: for each $s \subseteq \omega$, let $\sigma_s \in \Rhat$ be defined via $\sigma_s(m) = \epsilon_{s \cap m} + (r^m)$. Then $\{\sigma_s: s \subseteq \omega\}$ is a listing of $\Gamma$ without repetition.

 \begin{Lemma} \label{oneplus} Every non-zero $\gamma\in \Gamma$ is a central non-zero divisor.
 \end{Lemma}
 
 \begin{proof} That every $\gamma\in\Gamma$ is central is clear.
 As every non-zero $\gamma\in\Gamma$ is equal to $r^n(1+r\delta)$ for some $n\ge 0$ and $\delta\in \Rhat$, 
 it suffices to show that for every $\mu\in \Rhat$, if $\mu(1+r\delta)=0$, then $\mu=0$.  We will prove, by induction on $m\ge 1$, 
 that 
 
 \begin{quotation}  \noindent For any $\mu,\delta\in\Rhat$, if $\mu(1+r\delta)\in r^m\Rhat$, then $\mu\in r^m\Rhat$.
 \end{quotation}
 
 
  For $m=1$, assume $\mu(1+r\delta)\in r\Rhat$.  Then $\mu+\mu r \delta\in r\Rhat$, hence $\mu\in r\Rhat$.
 Assume the above holds for $m$ and suppose $\mu(1+r\delta)\in r^{m+1}\Rhat$. By the base case we have that $\mu \in r \hat{R}$. Choose $\nu$ such that $\mu=r\nu$.
 Plugging back in, $r\nu(1+r\delta)\in r^{m+1}\Rhat$, so $\nu(1+r\delta)\in r^m\Rhat$.
 Thus, $\nu\in r^m \Rhat$ by our inductive hypothesis, so $\mu=r\nu\in r^{m+1}\Rhat$, as required.  \qed
 \end{proof}
 
 There is a natural topology on $\Gamma$ formed by positing that $\{\O_{m, s}:s\subseteq m < \omega\}$ is a basis, where
 $$\O_{m, s}:=\{\gamma\in\Gamma:\gamma(m)= \epsilon_s\}$$
 for every $s \subseteq m < \omega$.  With this topology, $\Gamma$ is Polish; indeed, it is homeomorphic to Cantor space. Writing $\Gamma = \{\sigma_s: s \subseteq \omega\}$ as above, the homeomorphism in question is obtained by sending $\sigma_s$ to $s$.
 
 For each $n\ge 1$ we consider the polynomial ring $\Rhat[x_i:i<n]$, where the indeterminants $x_i$ commute with each other and with the coefficients.
 In what follows, we will only evaluate polynomials $p(\xbar)\in R[\xbar]$ on tuples $\abar$ that are central, so there is no ambiguity.
 
 
 \begin{Definition}  {\em  Suppose $R'$ is a subring of $\Rhat$. A subset $\Gamma_0\subseteq \Gamma$ is {\em $R'$-algebraically independent} if, for all $n\ge 1$, for all $p(x_1,\dots,x_n)\in R'[x_1,\dots,x_n]$,
 and for all distinct $\{\gamma_1,\dots,\gamma_{n}\}\subseteq \Gamma_0$, if $p(\gamma_1,\dots,\gamma_n)=0$, then $p=0$ (i.e., every coefficient of $p$ is zero).
 
  For polynomials $p(x)\in \Rhat[x]$ in one variable, 
 we say {\em $p$ is $0$ on $\O_{m,s}$} if $p(\gamma)=0$ for every $\gamma\in\O_{m,s}$, and we say {\em $p=0$ on a ball around 0} if $p$ is 0 on some $\O_{m,\emptyset}$.
 } 
 \end{Definition}
 
 
 The next few lemmas tell us that countably infinite $R$-algebraically independent sets exist.


 \begin{Lemma}  \label{ball}    Suppose $p(x)\in \Rhat[x]$ is 0 on a ball around 0.  Then $p$ is the zero polynomial, i.e., every coefficient of $p$ is zero.
 \end{Lemma}
 
 
 \begin{proof}  We argue by induction on the degree of $p$.  For $p$ constant this is clear.  Suppose towards a contradiction $p(x)$ has degree bigger than $0$, and $p(x) \equiv 0$ on $\O_{m,\emptyset}$ and we have checked all polynomials of smaller degree. Since $p(0) = 0$, the constant coefficient of $p(x)$ must be $0$, so we can write $p(x) = x q(x)$ where $q(x)$ is a polynomial of smaller degree. For the contradiction, it suffices to show $q(x) = 0$; by the inductive hypothesis, it suffices to show $q(x) \equiv 0$ on $\O_{m,\emptyset}$.
 
 We first check that $q(a)=0$ for all  nonzero $a \in \O_{m,\emptyset}$.  For any such $a$, we know $p(a) = 0$, i.e. $a q(a) = 0$.  We conclude $q(a) = 0$ since $a$ is not a zero divisor. 
 So to finish, we just need to show that $q(0) = 0$. But note that for every $n > m$, $q(r^n) = 0$. Since $q(r^n) - q(0) \in r^n \Rhat$ (by factoring), we get that $q(0) \in r^n \Rhat$. Since this holds for sufficiently large $n$, we get $q(0) = 0$ as desired. \qed
\end{proof}

\begin{Lemma}  \label{Baire}  Let $R'\subseteq \Rhat$ be any countable subring.  There is some $\gamma\in \Gamma$ such that $\{\gamma\}$ is $R'$-algebraically independent.
\end{Lemma}
\begin{proof}  We first show that the zero set $Z(p) := \{\gamma \in \Gamma: p(\gamma) =0\}$ is closed and nowhere dense for every non-zero $p\in R'[x]$.  It is easily checked that $Z(p)$ is closed for every $p\in \Rhat[x]$.
To see that $Z(p)$ is nowhere dense, by way of contradiction assume that some $\O_{m,s}\subseteq Z(p)$.
Let $q(x):=p(x-\epsilon_s)$. Since $\O_{m,s}\subseteq Z(p)$ we have
$\O_{m,\emptyset}\subseteq Z(q)$.  So, by Lemma~\ref{ball}(2), $q=0$.  Since the mapping $r(x)\mapsto r(x-\epsilon_s)$ is a ring automorphism of $\Rhat[x]$, this implies $p=0$, contradiction.

Now $\Gamma$ is Polish and $\{Z(p):p\in R'[x]\setminus\{0\}\}$ is a countable set of nowhere dense, closed sets.  By Baire Category, there is $\gamma\in\Gamma$ such that
$p(\gamma)\neq 0$ for every non-zero $p\in R'[x]$.  \qed
\end{proof}

\begin{Lemma}  \label{existindep}  There is an infinite  $R$-algebraically independent $\Gamma_0 \subseteq \Gamma$.
\end{Lemma}  

\begin{proof}  It suffices to show that any finite $R$-algebraically independent set $\{\gamma_i:i<n\}\subseteq \Gamma$ can be extended.  
Let $R'$ be the (countable) subring of $\Rhat$ generated by $R\cup\{\gamma_i:i<n\}$.
By Lemma~\ref{Baire}, choose $\gamma_n\in\Gamma$ such that $\{\gamma_n\}$ is $R'$-algebraically independent.
It is easily checked that $\{\gamma_i:i\le n\}$ is $R$-algebraically independent.  \qed
\end{proof}

 \medskip
 With our eye on encoding free-like tagged $R$-modules, we consider direct sums $\bigoplus_\omega R$ and $\bigoplus_\omega \Rhat$, respectively.
 Note that our identification of $R$ in $\Rhat$ extends, hence we consider $\bigoplus_\omega R$ to be an $R$-submodule of $\bigoplus_\omega \Rhat$.
As well,  the sequence $\<r^m(\bigoplus_\omega \Rhat)\rangle_{m < \omega} $ is a strictly descending sequence of submodules with zero intersection.
 For what follows, we require a weak form of purity.

 \begin{Definition}  {\em  A left $R$-submodule $M$ of $N$ is {\em $r$-pure} if, for every $a\in M$, if the equation $rx=a$ has a solution in $N$, then it has one in $M$.
 }
 \end{Definition}
 
 As $r$ is not a zero divisor in $\Rhat$, if $\aa\in \bigoplus_\omega \Rhat$
 and $rx=\aa$ has a solution in
 $\bigoplus_\omega\Rhat$, then the solution is unique.  Thus we get a partial inverse function $r^{-1}:r(\bigoplus_\omega \Rhat)\rightarrow \bigoplus_\omega \Rhat$,
 and an $R$-submodule  $G$ is $r$-pure if and only if it is closed under $r^{-1}$.  In particular, given any subset $X\subseteq \bigoplus_\omega \Rhat$, there is a unique smallest $r$-pure
 $R$-submodule containing $X$.
 
 Note that $\bigoplus_\omega R$ is $r$-pure in $\bigoplus_\omega \Rhat$, since $R$ is $r$-pure in $\Rhat$ (since $r R = (r) = r \Rhat \cap R$).
 

\begin{Lemma} \label{wd}   Suppose $G\subseteq \bigoplus_\omega \Rhat$ is an $r$-pure $R$-submodule,
$(\delta_i:i<n)$ are from $\Rhat$, and $(\aa_i:i<n)$ are from $G$.  For each $i<n$ and $m\ge 1$, choose $s_{i,m}\in \delta_i(m)$ (so $s_{i,m}\in R$).
Then $\sum_{i<n} \delta_i \aa_i=0$ in $\bigoplus_\omega \Rhat$ if and only if for every $m\ge 1$, $\sum_{i<n} s_{i,m}\aa_i\in r^m G$.
\end{Lemma}

\begin{proof}  Fix $m\ge 1$.  Note that 
$$\left(\sum_{i<n} \delta_i\aa_i\right)-\left(\sum_{i<n} s_{i,m}\aa_i\right)\in r^m\bigoplus_\omega \Rhat$$
Thus, if $\sum_{i<n} \delta_i \aa_i=0$ in $\bigoplus_\omega \Rhat$, then it follows that $\sum_{i<n} s_{i,m}\aa_i\in r^m \bigoplus_\omega \Rhat$.
By $r$-purity, this is in $r^mG$.
Conversely, if $\sum_{i<n} s_{i,m}\aa_i\in r^m G$, then from above, $\sum_{i<n} \delta_i \aa_i\in r^m\bigoplus_\omega\Rhat$.  As this holds for all $m\ge 1$,
$\sum_{i<n} \delta_i \aa_i=0$.  \qed
\end{proof}

We will only use the following lemma in the case where $G$ and $H$ both contain $\bigoplus_\omega R$; in that case, $\Rhat G = \Rhat H = \bigoplus_\omega \Rhat$.
\begin{Lemma} \label{Rhatextend} Suppose $G,H\subseteq \bigoplus_\omega \Rhat$ are both $r$-pure $R$-submodules and $f:G\rightarrow H$ is an
$R$-module isomorphism.  Then there is an $\Rhat$-isomorphism $f^*: \Rhat G \to \Rhat H$ extending $f$.
\end{Lemma}

\begin{proof}  Define $f^*$ by $f^*(\sum_{i<n}\delta_i\aa_i)=\sum_{i<n}\delta_if(\aa_i)$.  This is well-defined by Lemma~\ref{wd} and is visibly an $\Rhat$-isomorphism.  \qed
\end{proof}

\medskip

We are finally ready to prove Theorem~\ref{A}.

\smallskip

\begin{proof}    By Proposition~\ref{freelikeBC}, it suffices to 
 describe a Borel reduction from the class of free-like, tagged left $R$-modules to the class of left $R$-modules.
For this, by Lemma~\ref{existindep}, fix a countably infinite,  $R$-algebraically independent $\{\gamma_i:i\in\omega\}\subseteq\Gamma$.
To ease notation, we take as inputs $\Mbar=(\bigoplus_\omega R,M_n)_{n\in\omega}$ with the additional property that $M_0=\oplus_\omega R$.
To each such $\Mbar$:
\begin{quotation}
\noindent  Let $G(\Mbar)$ be the smallest $r$-pure $R$-submodule of $\bigoplus_\omega\Rhat$ containing $\bigoplus_\omega R\cup \bigcup_{n\in\omega}\gamma_n M_n$.
\end{quotation}

By Lemma~\ref{Rhatextend}, if  $f:\Mbar\cong \Mbar'$ as tagged $R$-modules, then the isomorphism lifts to an $R$-module isomorphism $f':G(\Mbar)\rightarrow G(\Mbar')$.  
For the converse, we will need the following lemma.

\begin{Lemma}\label{MkChar}
Suppose $\overline{M} = (\bigoplus_\omega R, M_n)_{n \in \omega}$ is a free-like tagged left $R$-module. Then for all $n < \omega$ and for all $\aa \in G(\Mbar)$, we have $\aa \in M_n$ if and only if $\gamma_n \aa \in G(\Mbar)$.
\end{Lemma}
\begin{proof}
Clearly if $\aa \in M_n$ then also $\gamma_n \aa \in G(\Mbar)$ (it is one of the generators). Conversely, suppose $\aa \in G(\Mbar)$ and also $\gamma_n \aa \in G(\Mbar)$. We can write $\aa = r^{-k}(\mathbf{b} + \sum_{m} \gamma_m \mathbf{c}_m)$ for some $k < \omega$, some $\mathbf{b} \in \bigoplus_\omega R$ and some $\mathbf{c}_m \in M_m$. Similarly write $\gamma_n \aa = r^{-k'}(\mathbf{b}' + \sum_m \gamma_m \mathbf{c}'_m)$. After mutiplying through by $r^{k+k'}$ we thus get

$$r^{k'} \left( \gamma_n \mathbf{b}  + \sum_m \gamma_m \gamma_n \mathbf{c}_m\right) = r^k \left( \mathbf{b}' + \sum_m \gamma_m \mathbf{c}'_m\right).$$ 

Applying algebraic independence at each coordinate, we get that each $\mathbf{c}_m = 0$, and that $\mathbf{b}' = 0$, and that $\mathbf{c}'_m  =0$ for all $m \not= n$, and that $r^k \mathbf{c}'_n = r^{k'} \mathbf{b}$, i.e. $\mathbf{b} = r^{-k'}(r^k \mathbf{c}'_n)$. Since $M_n$ is pure in $\bigoplus_\omega R$ (see Remark~\ref{freelikeImpliesPure}), which is $r$-pure in $\bigoplus_\omega \Rhat$, and since $r^k \mathbf{c}'_n \in M_n$, we get that $\mathbf{b} \in M_n$ also. Since $\aa = r^{-k} \mathbf{b}$, we are done.  \qed
\end{proof}

To finish, suppose $h: G(\Mbar) \cong G(\Mbar')$ as $R$-modules. Extend $h$ to an $\Rhat$-automorphism $\sigma$ by Lemma~\ref{Rhatextend}; then use the previous lemma to see that $\sigma$ carries each $M_n$ to $M'_n$. Since $M_0 = M_0' = \bigoplus_\omega R$, this implies that $\sigma$ restricts to a tagged $R$-module isomorphism from $\Mbar$ to $\Mbar'$.   \qed
\end{proof}

\subsection{Two orthogonal ideals}

The goal of this subsection is to prove the following result.

\begin{Theorem}  \label{B}  Suppose $R$ is a countable ring and $x,y\in R$ are central elements satisfying
\begin{enumerate}
\item  $(x)\cap (y)=0$; and
\item The two-sided ideal $I=\mbox{Ann}(x)+\mbox{Ann}(y)$ is proper, i.e., $1\not\in I$.
\end{enumerate}
Then the theory of left $R$-modules is Borel complete.
\end{Theorem}

\begin{proof}
	
 Let $S:=R/I$.  In light of Lemma~\ref{endo} it suffices to construct a Borel embedding from the class of left $S$-endomorphism structures $(V,T)$ to the class of left $R$-modules.
Given $(V,T)$, where $V$ is a left $S$-module and $T:V\rightarrow V$ is an endomorphism, we construe $(V,T)$ as being $R$-modules and $R$-endomorphisms
in the usual manner.  That is, for each $r\in R$, define (left) multiplication by $r:V\rightarrow V$ by $r(v)=(r+I)(v)$ for each $v\in V$.  It is easily checked that $V$ is a left $R$-module
and $T$ is an $R$-endomorphism. Define the left $R$-module
$$W:=V\oplus\bigoplus_V R$$
and let $\{\ee_v:v\in V\}$ be a basis for $\bigoplus_V R$.  Note that $W$ does not depend on $T$.  
Let $J$ be the submodule generated by
$$\{x\ee_v - v:v\in V\}\cup\{y\ee_v-T(v):v\in V\}$$
and let  $M(V,T)$ be the $R$-module $W/J$.  The map $(V,T)\mapsto M(V,T)$ is Borel, and it is evident that isomorphic $S$-endomorphism structures get mapped to isomorphic $R$-modules,
but the reverse is more involved.  We will show that there are 0-definable copies of both $V$ and the graph of $T$ in  $M(V,T)$.

It is useful to define the $R$-endomorphism $\eps:\bigoplus_V R \rightarrow V$ as
$$\eps (\sum_{v\in V} r_v\ee_v)=\sum_{v\in V} r_v v$$  Note that a typical element $\cbar\in M(V,T)$ is of the form
$\cbar=v+\cc+J$ for some $v\in V$ and some  $\cc\in \bigoplus_V R$.
We record the following facts.

\begin{Fact}  \label{basicFacts}

\begin{enumerate}
\item  $xy=0$, hence $x,y\in I$.
\item    $\bigoplus_V I\subseteq ker(\eps)$.
\item  For any $\cbar=u+\cc+J\in M(V,T)$, $x\cbar=\eps(\cc)+J$ and $y\cbar=T(\eps(\cc)) + J$.
\end{enumerate}
\end{Fact}

\begin{proof} 

(1)  Since $x,y$ are central, $xy\in (x)\cap (y)$, which we assumed to be 0.  Thus, $x\in \mbox{Ann}(y)\subseteq I$, and dually $y\in \mbox{Ann}(x)\subseteq I$.

(2)  As $V$ is an $R/I$-module, $rv=0$ for every $r\in I$.  Thus, for any $\cc\in \bigoplus_V I$, $\eps(\cc)=\sum_{v\in V} r_v v=0$ since each $r_v\in I$.

(3)  Write $\cc=\sum_{v\in V} r_v\ee_v$ with all but finitely many $r_v=0$.  Then $x\cbar=xv+\sum_{v\in V} r_v x\ee_v +J$.
Since $x\in I$ by (1) and $V$ is a left $S$-module, $xv=0$, and $x\ee_v\equiv v$ (mod $J$), hence
$x\cbar=(\sum_{v\in V} r_v v)+J= \eps(\cc)+J$.  Similarly, $y\cbar=yv+\sum_{v\in V} r_v y\ee_v +J$, and since $y\ee_v \equiv T(v)$ (mod $J$)
we have $y\cbar=\sum_{v\in V} r_v T(v) + J =T(\eps(\cc))+J$ by the linearity of $T$.  \qed
\end{proof}

\medskip
\noindent{\bf Claim 1.}  $V\cap J=0$.
\medskip

\begin{proof}  A typical element of ${\bf j}\in J$ is of the form
$${\bf j}=\sum_{v\in V} r_v(x\ee_v-v) \ +\ \sum_{v\in V} r'_v(y\ee_v-T(v))$$
Letting $\aa=\sum_{v\in V} r_v \ee_v$ and $\bb=\sum_{v\in V} r'_v\ee_v$ this becomes
$${\bf j}=x\aa+y\bb-\eps(\aa)-T(\eps(\bb))$$
Now suppose  ${\bf j}\in J$.  As $\eps(\aa),T(\eps(\bb))\in V$ we would have
$x\aa+y\bb\in V$.
Since $V$ is a direct summand, this implies $x\aa+y\bb=0$ in $\bigoplus_V R$.
However, our assumption that $(x)\cap(y)=0$ in $R$ implies that
$(x\bigoplus_V R)\cap (y\bigoplus_V R)=0$, hence $x\aa=y\bb=0$.
Computing this coordinate by coordinate implies both $\aa,\bb\in\bigoplus_V I$.
So $\eps(\aa)=\eps(T(\bb))=0$ by Fact~\ref{basicFacts}(2).
Combining all of these gives ${\bf j}=0$.  \qed
\end{proof}

It follows that the restriction of the canonical map $W\mapsto W/J=M(V,T)$ to $V$ is
an $R$-module isomorphism $\pi:V\rightarrow\{v+J:v\in V\}$.  [It is 1-1 by Claim 1 and is visibly onto.]
From these facts, we can consider both $\pi[V]$ and $\pi[{\rm graph}(T)]$ in the $R$-module $M(V,T)$.

\medskip
\noindent{\bf Claim 2.}  The image $\pi[V]=xM(V,T)$.
\medskip

\begin{proof} Both directions follow from Fact~\ref{basicFacts}(3). 
Given $v\in V$, we have $x(\ee_v+J)=v+J$,  
and conversely, for any $\cbar\in M[V,T]$, $x\cbar=\eps(\cc)+J\in \pi[V]$.  \qed
\end{proof}

\medskip
\noindent{\bf Claim 3.}  For any $v,w\in V$, $T(v)=w$ iff $M(V,T)\models\exists\cbar [x\cbar=\pi(v)\wedge y\cbar=\pi(w)]$.
\medskip
\begin{proof}  Choose any $v,w\in V$.   If $T(v)=w$, then taking $\cbar:=(\ee_v+J)$ works.
Conversely, suppose there is some $\cbar$ such that $x\cbar=\pi(v)=v+J$ and $y\cbar=\pi(w)=w+J$. Write $\cbar = u + \cc + J$.
Fact~\ref{basicFacts}(3) gives $\eps(\cc)+J=v+J$ and $T(\eps(\cc))+J=w+J$.
Thus, by Claim 1, $v=\eps(\cc)$ and $w=T(\eps(\cc))=T(v)$.  \qed
\end{proof}

This finishes the proof of Theorem~\ref{B}.  \qed
\end{proof}

\subsection{Descending chains of annihilator ideals}

Throughout this subsection, we concentrate on countable, commutative rings $R$.
\begin{Definition}  {\em  An ideal $I\subseteq R$ is an {\em annihilator ideal}  if for some $X \subseteq R$, we have $I=\mbox{Ann}_R(X)$, i.e. $I$ is the set of all $a \in R$ with $ax =0$ for all $x \in X$.
}
\end{Definition}

It is easily verified that for $I$ any annihilator ideal, $I=\mbox{Ann}_R(Ann_R(I))$, and in fact $J=\mbox{Ann}_R(I)$ is the largest subset of $R$ annihilated by $I$.
The goal of this subsection is to prove the following theorem.

\begin{Theorem}  \label{C}  
If a  countable, commutative ring $R$ has an infinite, descending sequence $(I_n:n\in\omega)$ of non-zero annihilator ideals satisfying 
$\bigcap_{n\in\omega} I_n=0$, then the theory of $R$-modules is Borel complete.
\end{Theorem}

Fix such ring $R$ and sequence $(I_n:n\in\omega)$ for the whole of this subsection.
As in the proof of Theorem~\ref{A}, we form a natural inverse limit ring and endow it with a natural Polish topology.
\begin{Definition}  {\em Let $\Rhat$ be the inverse limit ring $\varprojlim_{n\in\omega} R/I_n$.  Elements of $\Rhat$ are thus coherent $\omega$-sequences $\gamma$ such that $\gamma(n)\in R/I_n$
and, whenever $m\ge n$, $\pi_{m,n}(\gamma(m))=\gamma(n)$, where $\pi_{m,n}:R/I_m\rightarrow R/I_n$ is the canonical map.  We endow $\Rhat$ with a topology as follows:  Give each $R/I_n$ the 
discrete topology, given $\prod_{n\in\omega} R/I_n$ the product topology, and then give $\Rhat\subseteq \prod_{n\in\omega} R/I_n$ the subspace topology.  
}
\end{Definition}

As $\Rhat$ is a closed subspace of a Polish space, it is
Polish itself.  Also, as we can evaluate the ring operations coordinatewise, they are all continuous, so $\Rhat$ is a topological ring.  

Elements of $R/I_n$ correspond to cosets of $R$, so if $\gamma\in \Rhat$ and $n\in\omega$, picking $r\in\gamma(n)$ amounts to writing $\gamma(n)=r+I_n$ for some $r\in R$.
We identify a particular neighborhood basis of $0$ in $\Rhat$.  Namely, for each $n\in\omega$, let $\O_n:=\{\gamma\in \Rhat:\gamma(n)=0\}$.  Note that each $\O_n$ is a clopen ideal of $\Rhat$
and that $\bigcap_{n\in\omega}\O_n=0$.  Note that $\{\O_n+r:n\in\omega,r\in R\}$ forms a basis for the full topology on $\Rhat$.

Since $\bigcap I_n=0$, we can view $R$ as a subring of $\Rhat$.  Under this identification, for each $r\in R$ we view $r$ as the $\omega$-sequence defined by $r(n)=r+I_n$.  $R$ is dense in $\Rhat$, since always $r\in\O_n+r$.  We endow $R$ with the subspace topology from $\Rhat$.  Note then that $I_n=\O_n\cap R$ and $\{I_n:n\in\omega\}$ is a neighborhood basis of 0 in $R$.

For each $n\in \omega$, put $J_n:=\mbox{Ann}_R(I_n)$.  Note that then $J_n$ is also an annihilator ideal and $\mbox{Ann}_R(J_n)=I_n$.  Note also:

\begin{Lemma}  \label{closure}  For each $n\in\omega$, $\O_n=\mbox{Ann}_{\Rhat}(J_n)=\overline{I_n}$.
\end{Lemma}

\begin{proof}  It is clear that $\mbox{Ann}_{\Rhat}(J_n)$ contains $\mbox{Ann}_R(J_n)=I_n$ and is closed, so $\overline{I_n}\subseteq \mbox{Ann}_{\Rhat}(J_n)$.


Next, we show that $Ann_{\Rhat}(J_n)\subseteq\O_n$.  To see this, choose $\gamma\not\in \O_n$, i.e., $\gamma(n)\neq 0$.  For each $m\ge n$, choose $r_m\in\gamma(m)$.  Since $r_n\not\in I_n$,
there is $x\in J_n$ such that $r_nx\neq 0$.  Choose $m\ge n$ large enough so that $r_n x \not \in I_m$.  Since $r_m-r_n\in I_n$ by coherence, we have $r_mx-r_nx\in I_nx=0$, so $r_mx\not\in I_m$.  But then
$(\gamma x)(m)=r_mx+I_m\neq 0$, so $\gamma x\neq 0$, i.e., $\gamma$ does not annihilate $J_n$.

Finally, we show $\O_n\subseteq\overline{I_n}$.  But this is clear, since given $\gamma\in \O_n$, if we let $r_n\in\gamma(n)$, then $\{r_n:n\in\omega\}$ is a sequence from $I_n$ converging to $\gamma$.  \qed
\end{proof}

As in the proof of Theorem~\ref{A}, we pass to the direct sums $\bigoplus_\omega R$ and $\bigoplus_\omega \Rhat$.

Using Lemma~\ref{closure} we prove definability results that are analogous to Lemmas~\ref{wd} and \ref{Rhatextend}.

\begin{Lemma}  Suppose $(\delta_i:i<n)\in \Rhat^{\,n}$ are given and, for each $m$, choose $r_{i,m}\in\delta_i(m)$.
Then for all $(\aa_i:i<n)\in \bigoplus_\omega \Rhat$, we have
$$\sum_{i<n} \delta_i\aa_i=0\quad \hbox{if and only if $J_m\cdot(\sum_{i<n} r_{i,m}\aa_i)=0$ for every $m$}$$
\end{Lemma}
\begin{proof}  Proceeding coordinate-wise, it suffices to prove that for all $(a_i:i<n)$ from $\Rhat$, $\sum_{i<n} \delta_i a_i=0$ if and only if $J_m \cdot (\sum_{i<n} r_{i,m}a_i)=0$ for every $m$.
Note that for each $m$, $(\sum_{i<n}\delta_i a_i)(m)=(\sum_{i<n} r_{i,m}a_i)(m)$.  Since $\sum_{i<n}\delta_ia_i=0$ if and only if each $(\sum_{i<n} \delta_ia_i)(m)=0$,
we conclude that $\sum_{i<n}\delta_ia_i=0$ if and only if $(\sum_{i<n} r_{i,m}a_i)(m)=0$  for all $m$.  By Lemma~\ref{closure}, this holds if and only if
each $J_m\cdot(\sum_{i<n} r_{i,m}a_i)=0$.  \qed
\end{proof}

Following the argument (but without the assumption of $r$-purity) in \ref{Rhatextend}, we immediately obtain:

\begin{Lemma}  \label{auto}  Suppose $\bigoplus_\omega R\subseteq G,H\subseteq \bigoplus_\omega \Rhat$ are $R$-submodules and $f:G\rightarrow H$ is an
$R$-module isomorphism.  Then there is an $\Rhat$-automorphism $f^*$ of $\bigoplus_\omega \Rhat$ extending $f$.
\end{Lemma}

\begin{Definition}  {\em  If $p(x_1,\dots,x_n)\in \Rhat[x_1,\dots,x_n]$, let 
$Z(p):=\{\abar\in \Rhat^{\,n}: p(\abar)=0\}$ be its zero set.  Any such $Z(p)$ is closed in $\Rhat^{\,n}$.  
Given any $X\subseteq \Rhat^{\,n}$, let $\partial X$ denote its boundary, i.e., the  set of all $\abar\in \Rhat^{\,n}$ such that any open set containing $\abar$ intersects both
$X$ and its complement.  Any such $\partial X$ is closed and nowhere dense.

Say that a subset $\Gamma_0\subseteq \Rhat$ is {\em $R$-sufficiently generic} if, for all $n$,  for all sequences $(\gamma_1,\dots,\gamma_n)$ of distinct elements of $\Gamma_0$,
and for all $p(x_1,\dots,x_n)\in R[x_1,\dots,x_n]$, $(\gamma_1,\dots,\gamma_n)\not\in\partial Z(p)$.
}
\end{Definition}

As $R$ is countable and the topology on $\Rhat$ is Polish, the following lemma follows from the Baire category theorem.

\begin{Lemma} \label{existssuitably} An infinite sufficiently generic $\Gamma_0\subseteq \Rhat$ exists.
\end{Lemma}


Enumerate $\Gamma_0$ as $\{\gamma_{n,m}:n,m\in\omega\}$ and put $J:=\bigcup_{n\in\omega} J_n$.
To conclude that the theory of $R$-modules is Borel complete, we exhibit a Borel reduction from tagged $R/J$-modules to $R$-modules.
Fix any countable tagged $R/J$-module $\Vbar=(V,V_n)_{n\in\omega}$.  For notational simplicity later, assume that $V_0=V$ and $V_1 = 0$.
To construct the Borel embedding, first consider $\bigoplus_V R$ as an $R$-module with standard basis $\{\ee_v:v\in V\}$
and let $\eps:\bigoplus_V R\rightarrow V$ denote the $R$-module surjection given by $\eps(\sum_{v\in V} r_v)=\sum_{v\in V} r_v v$.
Let $G(\Vbar)$ be the $R$-submodule of $\bigoplus_V \Rhat$ generated by $\bigoplus_V R\cup\bigcup_{n,m} \gamma_{n,m} \eps^{-1}(V_n)$.

We argue that the map $\Vbar\mapsto G(\Vbar)$ is a Borel reduction.  The Borelness is clear, and if $\Vbar\cong \Vbar'$ as tagged $R/J$-modules, then the $R$-modules
$G(\Vbar)$ and $G(\Vbar')$ are isomorphic.  The reverse direction requires more work.  

Given any $\Vbar$, let $$J_*:=\{\gg\in G(\Vbar):\gg\ \hbox{annihilates some $\O_n$}\}$$
$J_*$ is evidently an $R$-submodule of $G(\Vbar)$ but more is true.

\begin{Fact}  \label{Jstar}
\begin{enumerate}
\item  $\Rhat J_*=J_*$, hence $J_*$ can be construed as an $\Rhat$-module.
\item  $J_*\cap \bigoplus_V R\subseteq \ker(\eps)$.
\end{enumerate}
\end{Fact}

\begin{proof}
(1) Choose any $\gamma\in \Rhat$ and $\jj\in J_*$.  Say $\jj$ annihilates $\O_n$.  Choose any $r\in \gamma(n)$.
Then as $\gamma-r\in \O_n$, $(\gamma-r)\jj=0$.  Thus, $\gamma \jj=r\jj\in J_*$ since $J_*$ is an $R$-module.

(2)  Choose any $\aa=\sum_{v\in V} r_v\ee_v\in J_*$.  Choose $n$ so that $\aa$ annihilates $\O_n$.  Choose any $i\in I_n$.
As $i\aa=0$, we have $\sum_{v\in V} (ir_v)\ee_v=0$.  Thus, for every $v\in V$, $ir_v=0$.  As this holds for every $i\in I_n$, each $r_v\in J_n\subseteq J$.
Thus, $\aa=\sum_{v\in V} j_v\ee_v$ with every coefficient in $J$.  But $V$ is an $R/J$-module, so every $j\in J$ annihilates $V$. 
Thus, $\eps(\aa)=\sum_{v\in V} j_v v=0$, i.e, $\aa\in\ker(\eps)$.  \qed
\end{proof}

Put $M:=\bigoplus_V R + J_*$ (not a direct sum!) and put $M_n:=\eps^{-1}(V_n)+J_*$.  (Note that as $V_0=V$ and $V_1 = 0$, we have $M_0=M$ and $M_1 = \mbox{ker}(\epsilon) + J_*$).

\begin{Fact}  \label{corr}  For each $n\in\omega$, $$M_n=\{\gg\in G(\Vbar):\gamma_{n,m}\gg\in G(\Vbar)\ \hbox{for all $m\in\omega$}\}$$
\end{Fact}

\begin{proof}  First suppose $\gg\in M_n$.  Say $\gg=\aa+\jj$ with $\aa\in \eps^{-1}(V_n)$ and $\jj\in J_*$ and choose any $\gamma_{n,m}\in \Gamma_0$.
Then $\gamma_{n,m}\aa$ is a generator of $G(\Vbar)$ and $\gamma_{n,m}\jj\in J_*\subseteq G(\Vbar)$ by Fact~\ref{Jstar}(1), so $\gamma_{n,m}\aa\in G(\Vbar)$.


Conversely, suppose $\gg\in G(\Vbar)$ has $\gamma_{n,m} \mathbf{g} \in G(\overline{V})$ for all $m \in \omega$.  Write $\gg=\aa+\sum_{n',m'}\gamma_{n',m'}\aa_{n',m'}$ with $\aa\in\bigoplus_V R$ and $\aa_{n',m'}\in\eps^{-1}(V_{n'})$ and all but finitely many $\mathbf{a}_{n',m'}=0$.
Fix $m^*\in\omega$ so that $\mathbf{a}_{n,m_*}=0$.  As we are assuming $\gamma_{n,m^*}\gg\in G(\Vbar)$, we can similarly write
$$\gamma_{n,m^*}\gg=\bb+\sum_{n',m'} \gamma_{n',m'}\bb_{n'm'}.$$
Rewriting, we have
$$-\mathbf{b} - \sum_{(n', m') \not= (n, m_*)} \gamma_{n', m'} \mathbf{b}_{n',m'} + \gamma_{n, m_*}\left(\mathbf{a}+ \sum_{n', m'} \gamma_{n', m'} \mathbf{a}_{n', m'}-\mathbf{b}_{n, m_*} \right) = 0.$$

Recall that all of this takes place within $\bigoplus_V \hat{R}$. By sufficient genericity applied at each coordinate, we get that there is some open $\mathcal{O}_k$ such that the above equation remains true whenever we perturb $\gamma_{n', m'}$ by anything in $\mathcal{O}_k$. By perturbing $\gamma_{n, m_*}$ and leaving fixed all other $\gamma_{n', m'}$ (and using $\mathbf{a}_{n, m_*} = 0$) we get that $\mathcal{O}_k \left(\mathbf{a}+ \sum_{n', m'} \gamma_{n', m'} \mathbf{a}_{n', m'}-\mathbf{b}_0 \right) = 0$, i.e. $\mathcal{O}_k(\mathbf{g} - \mathbf{b}_0) = 0$, i.e. $\mathbf{g} - \mathbf{b}_0 \in J_*$. Since $\mathbf{b}_0 \in \epsilon^{-1}(M_n)$ we are done.  \qed
\end{proof}

Next, extend the $R$-homomorphism $\eps$ to $\eps_*:M\rightarrow V$ by $\eps_*(\aa+\jj)=\aa$.  By Fact~\ref{Jstar}(2), $\eps_*$ is a well-defined $R$-homomorphism and is surjective because $\eps$ was. 
Let $K_*:=\ker(\eps_*)$.   It follows that $\eps_*$ induces an isomorphism between the tagged $R$-modules $(M/K_*,M_n/K_*)_{n \in \omega}$ and $\overline{V}$.

Recall that $V_0 = V$ and $V_1 = 0$.  Hence $M_0 = M$; also, an easy computation shows that $M_1 = K_*$. 

With the above results, we now finish. Suppose there is an $R$-module isomorphism $f:G(\Vbar)\rightarrow G(\Vbar')$, Then by Lemma~\ref{auto}, $f$ extends to an $\Rhat$-isomorphism, and by Fact~\ref{corr}, $f$ maps $M_n$ onto $M_n'$ for each $n$ (and hence maps $M$ to $M'$ and $K_*$ to $K_*'$). Thus $\overline{V} \cong (M/K_*, M_n/K_*)_{n \in \omega} \cong (M'/K'_*, M'_n/K'_*)_{n \in \omega} \cong \overline{V}'$. \qed

\section{Proof of Theorem~\ref{big}}   \label{TWO}

We conclude the paper by using the results above to prove Theorem~\ref{big}.  In light of Theorem~\ref{classical}  it suffices to show that if
$R$ is a commutative, countable ring for which the theory of $R$-modules is not Borel complete, then $R$ is an Artinian principal ideal ring.

\subsection{Gathering Ingredients}


In this section we collect together some algebraic facts and some corollaries of our previous theorems. First we recall some definitions for a commutative ring $R$:

\begin{Definition} \label{commalg} {\em 
\begin{enumerate}
	\item An ideal $\pp$ is {\em prime} if $R/\pp$ is an integral domain, i.e. $ab \in \pp$ implies $a \in \pp$ or $b \in \pp$.
\item An ideal $\m$ is {\em maximal} if $R/\m$ is a field, i.e. $\m$ is a maximal element of the poset of proper ideals of $R$. (Any maximal ideal is prime.)
\item An element $a \in R$ is {\em nilpotent} if $a^n = 0$ for some $n$. An ideal $I$ is nilpotent if $I^n = 0$ for some $n$. $I$ is {\em locally nilpotent} (or nil) if each element of $I$ is nilpotent.
\item The {\em Jacobson radical} of $R$ is the intersection of all maximal ideals of $R$.
\item The {\em  nilradical} of $R$ is the intersection of all prime ideals, or equivalently, the set of all nilpotent elements of $R$. This is contained in the Jacobson radical.
\item $R$ is {\em Artinian} if there are no infinite descending chains of ideals. $R$ is {\em Noetherian} if every ideal is finitely generated, i.e. there are no infinite ascending chains of ideals.
\item $R$ is a {\em principal ideal ring} if each ideal of $R$ is principal, i.e., is generated by a single element.
\item $R$ is {\em local} if it has a unique  maximal ideal.
\item An element $e \in R$ is {\em idempotent} if $e^2 = e$. Two idempotents $e, e'$ are {\em orthogonal} if $e e' = 0$.
\end{enumerate}
}
\end{Definition}

\begin{Lemma}[Chinese Remainder Theorem]
\label{CRT}   Suppose $R$ is a commutative ring and $(\m_i: i < n)$ is a sequence of distinct maximal ideals of $R$. Then the natural homomorphism $R \to \prod_{i < n} R/\m_i$ is surjective.
\end{Lemma}
\begin{proof}  This is a special case of Theorem 2.1 of \cite{Lang}.
\end{proof}

\begin{Lemma} [Hopkins-Levitsky]     \label{ArtinianImpliesNoetherian}
If $R$ is Artinian then $R$ is Noetherian. (This holds for all rings, not just the commutative case.)
\end{Lemma}
\begin{proof}
See for instance Theorem 4.15 of \cite{FC}.
\end{proof}

\begin{Lemma} \label{Ingred1}  The class of Artinian principal ideal rings is closed under finite direct products.
\end{Lemma}
\begin{proof}  Suppose $R= \prod_{i < n} R_i$ where each $R_i$ is Artinian and principal. Note that each ideal $J$ of $R$ can be written as $\prod_{i < n} J_i$ where $J_i$ is an ideal of $R_i$. Since each $J_i$ is principal, so must $J$ be; and since the finite product of well-founded linear orders is still well-founded, $R$ is Artinian. \qed
\end{proof}

\begin{Lemma} \label{Ingred2}
Suppose $R$ is a commutative ring with locally nilpotent Jacobson radical (i.e. the Jacobson radical is equal to the nilradical). Suppose further $R$ has only finitely many maximal ideals. Then $R$ is a finite direct product of local rings.
\end{Lemma}

\begin{proof}  Since $J(R)$ is locally nilpotent, by Lemma 21.28 of \cite{FC}, idempotents can be lifted modulo $J(R)$. Since $R$ has finitely many maximal ideals, $R/J(R)$ is a finite product of fields, hence is semisimple. By definition, this implies $R$ is semiperfect. By Theorem 23.11 of \cite{FC}, we get that $R$ is a finite direct product of local rings. \qed
\end{proof}

\begin{Lemma} \label{Ingred3}
Suppose $R$ is a commutative ring such that $R$-modules are not Borel complete. Then the following all hold:

\begin{enumerate}
	\item[a)] Every prime ideal of $R$ is maximal.
	\item[b)] The Jacobson radical of $R$ is equal to its nilradical, and hence is locally nilpotent.
	\item[c)] There is no infinite family of pairwise orthogonal idempotents.
\end{enumerate}

\end{Lemma}
\begin{proof}
		
(a) If $\pp$ were some prime ideal that is not maximal,  then $R/\pp$ would be an integral domain that is not a field.  By Theorem~\ref{A}, the theory of $R/\pp$-modules would be Borel complete (let $r \in R/\pp$ be any nonzero nonunit); but this contradicts Lemma~\ref{R/I}.  

(b) Follows immediately from (a).

(c) We prove the contrapositive. Let $\{e_n: n < \omega\}$ be the family of pairwise orthogonal idempotents. Let $I = \mbox{Ann}_R(e_n: n < \omega)$. Since each $e_n e_n \not= 0$, we have each $e_n \not \in I$, so we can mod out by $I$ without disturbing the hypotheses. We have then arranged that for all $r \in R$ there is some $n < \omega$ with $r e_n \not= 0$. For each $n < \omega$ let $I_n = \mbox{Ann}_R(e_m: m < n)$. Note that $e_{n+1} \in I_n \backslash I_{n+1}$, so this is a strictly descending chain of annihilator ideals with empty intersection. By Theorem~\ref{C}, $R$-modules are Borel complete.
 \qed
\end{proof}



\subsection{Infinitely Many Maximal Ideals}

In this subsection we prove the following theorem.
\begin{Theorem} \label{manymaximal}   Suppose $R$ is a countable, commutative ring that has infinitely many maximal ideals. Then the theory of $R$-modules is Borel complete.
\end{Theorem}

\begin{proof}  Fix such an $R$.

\medskip
\noindent{\bf Claim 1.}  There is a sequence $(\m_n:n < \omega)$ of distinct maximal ideals of $R$ such that for all $r \in R$, $\{n < \omega: r \in \m_n\}$ is either finite or cofinite.

\begin{proof}
Take $X$ to be a countably infinite set of maximal ideals of $R$ and let $\mathcal{U}$ be a nonprincipal ultrafilter on $X$. Enumerate $R = \{r_n: n < \omega\}$ and for each $n < \omega$, let $X_n$ be either $\{\m \in X: r_n \in \m\}$ or its complement, whichever is in $\mathcal{U}$. Inductively pick $\m_n$ so that $\m_n \in \left(\bigcap_{m \leq n} X_n\right) \backslash \{\m_m: m < n\}$. This easily works.  \qed
\end{proof}
	
	Fix such a sequence $(\m_n: n < \omega)$ for the rest of the proof. Note that we can mod out by $\bigcap_n \m_n$ without disturbing the hypotheses or conclusion (using Lemma~\ref{R/I}); i.e. we can arrange $\bigcap_n \m_n = 0$. It follows that the natural homomorphism $R \to \prod_n R/\m_n$ is injective, so we can view $R$ as a subring of $\prod_n R/\m_n$. Rephrased, for $r \in R$ and $n < \omega$, put $r(n) := r + \m_n \in R/\m_n$. 
	
	As notation, for each $r \in R$, put $\mbox{supp}(r) = \{n: r(n) \not= 0\}$. By arrangement we have that for all $r \in R$, $\mbox{supp}(r)$ is either finite or cofinite. Let $\pp$ be the set of all $r \in R$ whose support is finite; this is easily seen to be a prime ideal.
	
	
	For each $n$, let $e_n \in \prod_n R/\m_n$ be the idempotent defined via $e_n(n) = 1$, and $e_n(m) = 0$ for all $m \not= n$. Then $(e_n: n < \omega)$ is an infinite family of orthogonal idempotents. Let $u \subseteq \omega$ be the set of all $n$ such that $e_n \in R$. 
	
	\medskip
	\noindent
	{\bf Claim 2.}  
	For each $n \not \in u$, we have $\pp \subseteq \mathbf{m}_n$.

	\begin{proof}
		Rephrased, we want to show that if $r \in \pp$, then $\mbox{supp}(r) \subseteq u$. So let $n \in \mbox{supp}(r)$; it suffices to show $e_n \in R$. By the Chinese Remainder Theorem (see Lemma~\ref{CRT}), we can find some $s \in R$ such that $s(n) = r(n)^{-1}$,  and for all $m \in \mbox{supp}(r)$ with $m \not= n$, $s(m) = 0$. Then $rs = e_n$ so $e_n \in R$.  \qed
	\end{proof}
	
	Now if $u$ is infinite then $R$-modules are Borel complete by Lemma~\ref{Ingred3}(c), so we can suppose $u$ is finite. In that case, $(\mathbf{m}_n: n \not \in u)$ is an infinite family of distinct maximal ideals containing $\pp$, so in particular, $\pp$ is not maximal. By Lemma~\ref{Ingred3}(a), $R$-modules are Borel complete in any case.  \qed
\end{proof}

\subsection{The Local Case}

In this subsection we prove the following theorem.

\begin{Theorem}\label{LocalCase}
Suppose $R$ is a local commutative ring. If $R$-modules are not Borel complete, then $R$ is an Artinian principal ideal ring.
\end{Theorem}

\begin{proof}
 Fix a  local commutative ring $R$ with maximal ideal $\m$, such that $R$-modules are not Borel complete.
Note that $\m$ is locally nilpotent, by Lemma~\ref{Ingred3}(b).
%
%
%

\medskip
\noindent{\bf Claim 1.}
	The ideals of $R$ are linearly ordered under inclusion.

\begin{proof} 
By way of contradiction, suppose there were ideals $I,J$ with $I\not\subseteq J$ and $J\not\subseteq I$. Necessarily, $I, J \subseteq \m$. We can mod out by $I \cap J$ without disturbing the hypotheses (using Lemma~\ref{R/I}); so suppose $I \cap J = 0$.

Choose $r\in I\setminus J$ and $s\in J\setminus I$.  Since the principal ideal $(s)\subseteq J$,
we have $(r)\not\subseteq (s)$ and dually, $(s)\not\subseteq (r)$. By arrangement we have $(r) \cap (s) = 0$.  But note that $\mbox{Ann}(r) \subseteq \m$ and $\mbox{Ann}(s) \subseteq \m$, so by Theorem~\ref{B}, the theory of $R$-modules is Borel complete, contradicting our hypothesis on $R$.  \qed
\end{proof}

\medskip
\noindent{\bf Claim 2.}   The following both hold.
\begin{itemize}
\item[a)] If $J\subsetneq I$ are ideals, then $J\subsetneq (a)\subseteq I$ for every $a\in I\setminus J$: and 
\item[b)]  For any ideal $I$, there is no ideal $J$ with $\m I \subsetneq J \subsetneq I$.
\end{itemize}

\begin{proof} (a)  Choose any $a\in I\setminus J$.  Clearly, $(a)\subseteq I$.  However, since $a\not\in J$, $(a)\not\subseteq J$, so
$J\subseteq (a)$ by linearity.

(b) Note that this is trivial if $\m I = I$, so assume $\m I\subsetneq I$.
 After modding out we can suppose $\m I = 0$, and so $I$ is naturally an $R/\m$-vector space; subspaces of $I$ are the same as subideals. By Claim 1, the subspaces of $I$ are linearly ordered under inclusion. 
This implies that $\mbox{dim}(I) =1$, since if $a$ and $b$ are linearly independent, then $(a)$ and $(b)$ contradict linearity. It follows that $I = (a)$ for every nonzero $a \in I$, so there are no $J$ with $0 \subsetneq J \subsetneq I$, as desired. \qed
\end{proof}

\medskip
\noindent{\bf Claim 3.}  
$\m$ is finitely generated.

\begin{proof}
Suppose otherwise. Then $R$ is not Noetherian, so by Lemma~\ref{ArtinianImpliesNoetherian}, $R$ is not Artinian. Choose a strictly decreasing sequence $(J_n:n\in\omega)$ of ideals
and put $J:=\bigcap_{n\in\omega} J_n$. Note that we can mod out by $J$ without disturbing the hypotheses (using Lemma~\ref{R/I}); i.e. we can arrange $R$ has no minimal nonzero ideal. Note that this implies $\mbox{Ann}_R(\m) = 0$, by Claim 2b).

Since $\m$ is countable but not finitely generated,  repeatedly using Claim 2a), we get a strictly increasing sequence $((r_n):n\in\omega)$ of principal ideals contained in $\m$
with $\m=\bigcup_{n\in\omega} (r_n)$.  For each $n$, put $I_n:=\mbox{Ann}_R(r_n)$.  Clearly, $I_{n+1}\subseteq I_n$ for each $n$.  Now each $r_n$ is nilpotent (and we may assume non-zero); let $k \geq 2$ be least with $r_n^k = 0$. Then $r_n^{k-1}\in \mbox{Ann}_R(r_n)$, so $I_n\neq 0$. Also, $\bigcap_{n\in\omega} I_n = \mbox{Ann}_R(\m) = 0$, as remarked above. Thus, we have a descending sequence of annihilator ideals with zero intersection, so the theory of $R$-modules is Borel complete by Theorem~\ref{C}, contrary to hypothesis. \qed
\end{proof}

\medskip
\noindent{\bf Claim 4.} 
$\m$ is principal.
\begin{proof}
By the preceding claim, we can write $\m = (a_i: i < n)$ for some $a_i \in \m$. By linearity, we can suppose $(a_i) \subseteq (a_j)$ for $i \leq j$. But then $\m = (a_{n-1})$. \qed
\end{proof}

We can now finish. Write $\m = (x)$. Let $k$ be least with $x^k = 0$; so $k \geq 1$. Write $x^0 = 1$. Then $\langle (x^i):i  \leq k\rangle$ is a strictly descending sequence of principal ideals. Moreover, note that each $(x^{i+1}) = \m (x^i)$. By Claim 2b), $\langle (x^i): i \leq k \rangle$ enumerates all of the ideals of $R$, and so $R$ an Artinian principal ideal ring.
\qed
\end{proof}

\subsection{Putting it all together}
We now Prove Theorem~\ref{big}. So suppose $R$ is a commutative ring, such that the theory of $R$-modules is not Borel complete. We aim to show $R$ is an Artinian principal ideal ring.

By Theorem~\ref{manymaximal}, $R$ has only finitely many maximal ideals, and by Lemma~\ref{Ingred3}(b), the Jacobson radical of $R$ is locally nilpotent. By Lemma~\ref{Ingred2}, $R$ is a finite product of local rings $R = \prod_{n < n_*} R_n$. By Lemma~\ref{R/I}, for each $n$, the theory of $R_n$-modules is not Borel complete. By Theorem~\ref{LocalCase}, each $R_n$ is an Artinian principal ideal ring. By Lemma~\ref{Ingred1}, $R$ is an Artinian principal ideal ring.
%
%
%

\appendix

\section{Appendix:  Constructing limit objects with rich automorphism groups}


In this appendix, we describe suitable families of classes $\K$ of algebraic objects that have a limit structure with a rich automorphism group.
Our method is an extension of the notion of {\em merging} that appears in \cite{Herrings}.

\begin{Definition}\label{GrowingDef}  {\em  Let $\L$ be a countable language with a distinguished unary predicate $X$.  
A class $\K$ of $\L$-structures  is {\em $X$-growing} if every $A\in\K$ has an extension $B\supseteq A$ in $\K$ with $X^B$ properly extending
		$X^A$. 
		
	A class 	$\K$ of countable $\L$-structures  is {\em suitable} if $\K$ is closed under isomorphism, has disjoint amalgamation, is $X$-growing, and there is some $A \in \K$ with $X^A=\emptyset$.

		An $L$-structure $M$ is a {\em $\K$-limit} if $M=\bigcup\{A_n:n\in\omega\}$ with each $A_n\in\K$ and $A_n\subseteq A_m$ whenever $n\le m$.
	}
\end{Definition}

Our main theorem is the following:

\begin{Theorem}\label{GeneralSuitable}
	Suppose a class $\K$ is suitable.  
	Then there is some $\mathbf{K}$-limit $M$ admitting an equivalence relation $E$ on $X^M$ with infinitely many classes, such that each $h \in \mbox{Sym}(X^M/E)$ lifts to an automorphism of $M$.
\end{Theorem}

The following simple corollary might also be useful in some settings.

\begin{Corollary}
	Suppose $\mathbf{K}$ is a class of countable structures, is closed under isomorphism, has disjoint amalgamation and is $X$-growing.
	Then there is some $\mathbf{K}$-limit $M$ admitting some $Y \subseteq X^M$ and some equivalence relation $E$ on $Y$ with infinitely many classes, such that each $h \in \mbox{Sym}(Y/E)$ lifts to an automorphism of $M$.
\end{Corollary}
\begin{proof}
	Let $\mathbf{K'}$ be obtained from $\mathbf{K}$ by adding a new unary predicate $Y$ and insisting only that it is interpreted as a subset of $X$; apply the preceding theorem.  \qed
\end{proof}

\subsection{Extending and amalgamating equivalence relations}

We collect some simple remarks.
It is obvious that given any equivalence relation $(I,E^I)$ and any set $J\supsetneq I$, there are (many) equivalence relations $E^J$ on $J$ 
so that $(J,E^J)\supseteq (I,E^I)$ (note that this is a stronger statement than  $E^J\supseteq E^I$). 


Now suppose $(I,E^I)$ is a substructure of both $(J,E^J)$ and $(K,E^K)$ with $J\cap K=I$ and we investigate equivalence relations $E^*$ on $J\cup K$ extending
$E^J\cup E^K$.  Transitivity requires $E^*(x,y)$ to hold whenever $x/E^J$ and $y/E^K$ extend the same $E^I$-class. 
Recognizing this, call a bijection $\ell:J/E^J\rightarrow K/E^K$ {\em permissible} if for all $x \in I$ we have $\ell(x/E^J) = x/E^K$.

\begin{Lemma}  \label{permit}  Suppose $(I,E^I)$ is a substructure of both $(J,E^J)$ and $(K,E^K)$ with $J\cap K=I$ and $\ell:J/E^J\rightarrow K/E^K$ is a permissible bijection.
Then there is a unique equivalence relation $E_\ell$ on $J\cup K$ extending $E^J\cup E^K$ satisfying $E_\ell(x,y)$ whenever $\ell(x/E^J)=y/E^K$.
\end{Lemma}

\begin{proof}  Put $E_\ell:=E^J\cup E^K\cup\{(x,y)\in J\times K: \ell(x/E^J)=y/E^K\}$.
\qed
\end{proof}

\begin{Definition}  {\em  Given an equivalence relation $(J,E^J)$, a partial function $f:J\rightarrow J$ is {\em $E^J$-preserving} if, for all $x,y\in \mbox{dom}(f)$, $E^J(x,y)$ iff $E^J(f(x),f(y))$.
Such an $f$ induces a partial map $f/E^J:J/E^J\rightarrow J/E^J$. When we have $(J, E^J) \subseteq (K, E^K)$ then we like to identify $f/E^J$ with $f/E^K$; we write $f/E$ for both.
}
\end{Definition}

\subsection{General construction of structures with infinitary indiscernibles}

In this section, we prove Theorem~\ref{GeneralSuitable} in the special case when $\mathbf{K}$ is {\em strongly suitable}  (see below).   Actually, in our applications, this special case would be enough.

\begin{Definition}
{\em
Suppose $\L$ is a countable language with a distinguished unary predicate $X$.  (We allow multi-sorted languages, and we allow sorts to be empty.) 
For an $\L$-structure $A$ and $I\subseteq A$, let $\mbox{cl}^A(I)$ denote the smallest substructure of $A$ containing $I$, which may be empty.
We consider classes $\K$ of countable $\L$-structures.
Call $\K$ {\em unbounded} if every $A \in \mathbf{K}$ has a proper extension $B \supsetneq A$ in $\mathbf{K}$, and say $\K$ has
{\em disjoint amalgamation} if for all $A, B, C \in \mathbf{K}$ with $A \subseteq B, C$ and $B \cap C = A$, there is some $D \in \mathbf{K}$ with $B, C \subseteq D$. 
%
%
%
}
\end{Definition}  

\begin{Definition}  \label{suitabledef}
{\em  A class $\K$ is {\em strongly suitable} if

\begin{enumerate}
\item $\mathbf{K}$ is closed under isomorphism;
\item  For every $A\in \K$, 
 $X^A$ is finite and {\em generates $A$}, i.e., $A = \mbox{cl}^A(X^A)$;
 \item  $\K$ is unbounded; 
 \item  There is (at least one) $A\in \K$ with $X^A=\emptyset$; and
\item $\mathbf{K}$ has disjoint amalgamation.
\end{enumerate}
}
\end{Definition}

Clearly, $\K$ strongly suitable implies $\K$ suitable.  In particular, every structure in $\K$ is countable as it is generated by a finite set  in a countable language.

%

%

\begin{Theorem}\label{Suitable}
If $\mathbf{K}$ is strongly  suitable then there is some   $\mathbf{K}$-limit $M$ admitting an equivalence relation $E$ on $X^M$ 
with infinitely many classes, such that each $h \in \mbox{Sym}(X^M/E)$ lifts to an automorphism of $M$. 
\end{Theorem}


We note that with a little extra work, one can also arrange that the $\mathbf{K}$-limit $M$ is homogeneous, i.e. for all $A \subseteq M$ with $A \in \mathbf{K}$, every embedding $f: A \to M$ extends to an automorphism of $M$. But we won't need this.

\begin{proof} Suppose $\mathbf{K}$ is a strongly suitable class of $\L$-structures.
Let $\L_*:=\L\cup\{E\}$, where $E$ is a new binary relation symbol and let
  $\mathbf{K}_*$ consist of all expansions $A=(A_0,E)$ with $A_0\in\K$ and $E$ interpreted as an equivalence relation on $X^{A}$.  
  


%


\begin{Lemma} \label{zero}
	Suppose $A' \subseteq B, B'$ are in $\mathbf{K}_*$ with $B \cap B' = A'$. Suppose $g: B \cong B'$ is an isomorphism (which need not fix $A'$). Suppose $h \in \mbox{Sym}(X^B/E^B)$ satisfies that for all $a \in A'$, $h(g^{-1}(a)/E^B) = a/E^B$. Then there is some amalgam $C$ of $B$ and $B'$ over $A$, such that if we compute $g/E$ in $C$ then we get $g/E = h$.
\end{Lemma}
\begin{proof}
Let $B_0, B'_0, A'_0$ be the reducts of $B, B', A'$ to $\mathbf{K}$. 

By disjoint amalgamation in $\mathbf{K}$, we can find $C_0 \in \mathbf{K}$ with $B_0, B'_0 \subseteq C_0$. 
Now, to obtain an amalgam $C$ of $B$ and $B'$ over $A'$, it is enough to specify $E^C$ amalgamating $E^B$ and $E^{B'}$.  We obtain $E^C$ in two steps; first we find the (unique) equivalence relation
$E^*$ on $X^B\cup X^{B'}$ extending $E^B\cup E^{B'}$ satisfying $E^*(g(x),y)$ for every $x\in X^{B}$ and $y\in h(x/E^B)$.  Then, given $E^*$, any equivalence relation $E^C$ on $X^{C_0}$ extending $E^*$ will work.

By Lemma~\ref{permit}, finding such an equivalence relation $E^*$ on $X^B\cup X^{B'}$ amounts to showing that the bijection $\ell:X^B/E^B\rightarrow X^{B'}/E^{B'}$ given by
$\ell(h(x/E^B))=g(x)/E^{B'}$ is permissible. We verify this. $\ell$ is well-defined since $g$ preserves $E$. Choose any $a \in A'$; we need to show $\ell(a/E^B) = a/E^{B'}$. We have $a/E^B = h(g^{-1}(a)/E^B)$ by hypothesis. So by definition of $\ell$, we have $\ell(a/E^B) = g(g^{-1}(a))/E^{B'} = a/E^{B'}$ as desired. \qed
\end{proof}

Rephrasing:
\begin{Lemma}  \label{one}
Suppose $A \subseteq B$ are in $\mathbf{K}_*$ and $f: A \to B$ is an $\L_*$-embedding and $h \in \mbox{Sym}(X^B/E)$ extends $f/E$. 
Then we can find $C \supseteq B$ in $\mathbf{K}_*$ and an $\L_*$-embedding $g: B \to C$ extending $f$, with $g/E = h$.
\end{Lemma}
\begin{proof}
Write $A' = f[A]$. Choose $B' \in \mathbf{K}_*$ with $A' \subseteq B'$ and $B' \cap B = A'$, such that there is an $\L_*$-isomorphism $g: B \cong B'$ extending $f$. 
It is enough to show the hypotheses of the previous lemma are met. For this it suffices to show that for all $a \in A'$, $h(g^{-1}(a)/E^B) = a/E^B$. Note $g^{-1}(a) = f^{-1}(a)$; since $h$ extends $f/E$, we have $h(g^{-1}(a)/E^B) = h(f^{-1}(a)/E^B) = f(f^{-1}(a))/E^B=a/E^B$ as desired.
%
%
\qed
\end{proof}

\begin{Lemma}  \label{two}  For every $A\in\K_*$ there is $B\in \K_*$ with $B\supseteq A$ such that $X^B$ has at least one new $E$-class.
\end{Lemma}

\begin{proof}  As $\K$ is suitable, choose $B_0\in\K$ properly extending the $L$-reduct $A_0$ of $A$.  As $A_0=cl^{A_0}(X^{A})$ and $B_0=\cl^{B_0}(X^B)$, it follows that
$X^B$ properly extends $X^A$.  Choose any equivalence relation $E^B$ on $X^B$ extending $E^A$ with at least one new class.  Then $B=(B_0,E^B)$ works.
\qed
\end{proof}


%

\begin{Definition}  {\em
Suppose $B \subseteq C$ are in $\mathbf{K}_*$. Then say that $C$ is $B$-big if $X^C/E^C \supsetneq X^B/E^B$
 and for all $A \subseteq B$ in $\mathbf{K}_*$, for all embeddings $f: A \to B$ and for all $h \in \mbox{Sym}(X^B/E)$ extending $f/E$, there is an embedding $g: B \to C$ extending $f$ with $g/E = h$. 
}
\end{Definition}

\begin{Lemma}
Suppose $B \in \mathbf{K}_*$. Then there is $C \supseteq B$ in $\mathbf{K}_*$ which is $B$-big.
\end{Lemma}
\begin{proof}
For any $B\in \K_*$, being $B$-big describes finitely many constraints on an extension. 
We can handle each of these constraints by an application of either  Lemma~\ref{one} or  Lemma~\ref{two}.
Thus, via disjoint amalgamation, we can construct a finite chain in $\K^*$ to satisfy  them all. \qed
\end{proof}

To finish the proof of Theorem~\ref{Suitable},  recursively build a sequence $(A_n: n < \omega)$ from $\mathbf{K}_*$ such that $X^{A_0}=\emptyset$
 and each $A_{n+1}$ is $A_n$ big. Let $M_* = \bigcup_n A_n$, a $\mathbf{K}_*$-limit. and let $M$ be the reduct to $\mathbf{K}$. Write $E = E^{M_*}$. We claim $(M, E)$ works. Clearly $E$ has infinitely many classes; we just need to verify the lifting property. Let $h \in \mbox{Sym}(X^M/E)$ be given. Let $\mathcal{F}_h$ be the set of all triples $(A, A', f)$ where $A, A' \subseteq M_*$ are in $\mathbf{K}_*$, and $f: A \cong A'$ satisfies $f/E \subseteq h$.   We verify that $\mathcal{F}_h$ is a back-and-forth system of $L_*$-isomorphisms.
Clearly,  $\mathcal{F}_h$ is nonempty since $(A_0,A_0, \mbox{id}) \in \mathcal{F}_h$. By symmetry, it suffices to check that $\mathcal{F}_h$ is a forth system. So suppose $(A, A', f) \in \mathcal{F}_h$ and $A \subseteq B$; we want to find $(B', f')$ with $(B, B', f') \in \mathcal{F}_h$ extending $(A, A', f)$. Choose $n$ large enough so that $A, A', B \subseteq A_n$ and such that for all $x \in X^B$, $h(x/E)$ has a representative in $A_n$. Let $h'$ denote the restriction of $h$ to $X^B/E$.  Since $A_n$ contains a representative of $h'(x/E)$ for all $x\in X^B$,
 $h'$ is a finite partial injection from $X^{A_n}/E^{A_n}$ to itself. Let $h''$ be an arbitrary extension of $h'$ to $\mbox{Sym}(X^{A_n}/E^{A_n})$. Since $A_{n+1}$ is $A_n$-big, we can find some 
 $L_*$-isomorphism $g: A_n \to A_{n+1}$ extending $f$ with $g/E = h''$.   Let $f':=g\restriction B$ and put $B':=g[B]$.
Then $(B,B',f')\in\mathcal{F}_h$ and extends $(A,A',f)$. \qed
\end{proof}

\subsection{From strongly suitable to suitable}


We now finish the proof Theorem~\ref{GeneralSuitable}. Let $\mathbf{K}$ be suitable. By adding a dummy sort containing just a named constant, we can suppose that every structure in $\mathbf{K}$ has nonempty domain.

	Let $\mathcal{L}'$ be the two-sorted language with sorts $U$ and $V$; let it contain $\mathcal{L}$, considered as operating on the $U$-sort, and additionally, let it contain a binary relation symbol $R \subseteq U \times V$, and also  $n$-ary functions $f_{n, m}: V^n \to U$ for each $n, m < \omega$ (so when $n = 0$ we have countably many constants for elements of $U$). Let $\mathbf{K}'$ consist of all structures $A$ satisfying:

\begin{itemize}
\item  $V$ is finite; 
	\item $R^A$ is (the graph of) a surjective function from $X^A$ onto $V$;
	\item $A = \mbox{cl}^A(V)$;
	\item The reduct of $A$ to $U$ is an element of  $\mathbf{K}$.
\end{itemize}

Easily, $(\mathbf{K}', V)$ is strongly suitable.   Let $(M', E')$ be as given by Theorem~\ref{Suitable}.
Let $M$ be the reduct of $M'$ to $U$ and let $E$ be the equivalence relation on $X^M$ given by: $x E y$ if and only if $R(x) E' R(y)$. This works.


\begin{thebibliography}{99}


\bibitem{Herrings}  J.T.\ Baldwin, S.D.\ Friedman, M.\ Koerwien, and M.C.\ Laskowski,  Three red herrings around Vaught's conjecture. 
{\it Trans. Amer. Math. Soc.} {\bf  368} (2016), no. 5, 3673--3694.


\bibitem{CK}  I.\ S.\ Cohen and I.\ Kaplansky, Rings for which every module is a direct sum of cyclic modules. 
{\em Math.\  Zeitschrift}  {\bf 54}  (1951), 97--101.

\bibitem{FS} H.\ Friedman and L.\ Stanley,
A Borel reducibility theory for classes of countable structures,
{\em Journal of  Symbolic Logic} {\bf  54} (1989),  no. 3, 894--914.

\bibitem{GS} R.\ G\"obel and S. Shelah,  Absolutely indecomposable modules. {\em Proc. Amer. Math. Soc.}  {\bf 135} (2007), no. 6, 1641--1649. 


\bibitem{Hjorth}  G.\ Hjorth, The isomorphism relation on torsion-free abelian groups, {\em Fund. Math.} {\bf 175} (2002) 241--257.


\bibitem{FC}  T.\ Y.\ Lam, A first course in noncommutative rings,  Second edition. 
Graduate Texts in Mathematics, {\bf 131} Springer-Verlag, New York, 2001.

\bibitem{Lang}  S.\ Lang, {\em Algebra,} Revised third edition, 
Graduate Texts in Mathematics, {\bf 211} Springer-Verlag, New York, 2002.




\bibitem{PS}  G. Paolini and S. Shelah, 	Torsion free abelian groups are Borel complete,  [24 Feb 2021,   revised 16 June 2021, 25 Feb 2022, 5 April 2022, 25 May 2022.]
arXiv:2102.12371

\bibitem{Prest}  M. Prest, {\em Model theory and modules,}  London Mathematical Society Lecture Note Series, {\bf 130} Cambridge University Press, Cambridge, 1988.

\bibitem{ShU} D. Ulrich and S. Shelah,  
Torsion-Free Abelian Groups are Consistently a$\Delta^1_2$-complete, 
{\it Fund. Math.}  {\bf  247} (2019), no. 3, 275--297.



\end{thebibliography}
\end{document}